\documentclass[11pt,a4paper]{article}
\usepackage{geometry, empheq}
\usepackage[english]{babel}
\usepackage{amsmath}
\usepackage{amssymb}
\usepackage{amsthm}
\usepackage{xcolor,graphicx}
\usepackage{color}
\usepackage{varioref}
\usepackage{nicefrac}
\usepackage{caption}
\usepackage{subcaption}
\usepackage{float}
\usepackage[normalem]{ulem}
\usepackage{multirow}
\usepackage{enumerate}
\usepackage{algorithm}
\usepackage{algorithmic}

\newcommand{\arrow}{\rightarrow}
\newcommand{\ds}{\displaystyle}

\newcommand{\kw}{\rule{2mm}{2mm}}

\newtheorem{example}{\textbf{Example}}
\newcommand{\norm}[1]{{\|  #1\|}}
\renewcommand{\l}{\left}
\renewcommand{\r}{\right}
\newcommand{\om}{\Omega}
\newcommand{\ga}{\Gamma}
\newcommand{\fsparse}{\Upsilon}
\newcommand{\reals}{\mathbb{R}}
\newcommand{\sign}{\mathrm{sign}}

\renewenvironment{proof}{{Proof.}}{\hfill\kw}

\newcounter{theassumption}
\newtheorem{proposition}{{Proposition}}

\newtheorem{lemma}{{Lemma}}
\newtheorem{remark}{{Remark}}
\newtheorem{corollary}{{Corollary}}
\newtheorem{theorem}{{Theorem}}

\title{A difference--of--convex functions approach for sparse PDE optimal control problems with nonconvex costs\footnote{$^*$This research has been supported by Research Project PIJ-15-26 funded by Escuela Polit\'ecnica Nacional, Quito--Ecuador. Moreover, we acknowledge partial support of SENESCYT--MATHAmSud project SOCDE ``Sparse Optimal Control of Differential Equations''.}
}

\author{Pedro Merino$^\ddag$ \\
{pedro.merino@epn.edu.ec}\\
{\small $^\ddag$Research Center of Mathematical Modeling (MODEMAT)} \\{\small and Department of Mathematics, Escuela Polit\'ecnica Nacional}\\{\small Quito, Ecuador}
}
\date{last update 04/2019}

\begin{document}
\maketitle
\begin{abstract}
We propose a local regularization of elliptic optimal control problems which involves the nonconvex $L^q$ quasi--norm penalization in the cost function. The proposed \emph{Huber type} regularization allows us to formulate the PDE constrained optimization instace as a DC programming problem (difference of convex functions) that is useful to obtain necessary optimality conditions and tackle its numerical solution by applying the well known DC algorithm used in nonconvex optimization problems. By this procedure we approximate the original problem in terms of a consistent family of parameterized nonsmooth problems for which there are efficient numerical methods available. Finally, we present numerical experiments to illustrate our theory with different configurations associated to the parameters of the problem.
\end{abstract}


\section{Introduction}

Several sparse optimal control problems governed by PDEs have been considered in recent years. One of the pioneer works on this subject \cite{stadler09} introduced optimal control problems with $L^1$--norm penalization in order to promote sparse optimal solutions. These solutions are characterized by having small supports, which are interpreted as a ``localized'' action of the optimal control. This particular feature of sparse optimal controls is relevant in applications since it is rather difficult in practice to implement optimal controls distributed on the whole domain, which is the usual case of optimal control problems involving the Tikhonov regularization in the  $L^2$--norm in its cost functional. 

Another interesting class of optimal control problems involving sparsity were considered in \cite{caskun2016} and \cite{cclakun2013} where the set of feasible controls is chosen in the space of regular Borel measures. Therefore, optimal controls can be supported in a set of zero Lebesgue measure. A complete review on this subject, including parabolic problems, can be found in \cite{casas2017}. 

A less explored approach that offers sparse solutions induced by a penalization term was considered in \cite{itoku2014} which refers to penalizations consisting in nonconvex $L^q$ quasinorms with $q \in [0,1)$. This kind of penalizations has many important applications, for instance: in inverse problems on the reconstruction of the sparsest solution in undetermined systems \cite{ramlau2012},  image restoration \cite{hint2013}, compressive sensing \cite{foucart2013} and optimal control problems \cite{itoku2014}.

In particular, the limit case corresponding to $L^0$ penalization is a difficult problem which corresponds to the selection of the most representative variables of the optimization process, extending the notion of cardinality of the control variable in finite dimensions, represented by the $\ell^0$ norm, which is well known to be an NP--hard problem. $L^q$ quasinorms with $q \in (0,1)$ on the other hand, are a natural approximation to $L^0$ penalizations. However, they are neither convex nor differentiable. 

In \cite{itoku2014} a similar problem is considered involving a penalization term for the control variable involving the $H_0^1$--norm. This allows to get an explicit optimality system that can be solved directly by semi--smooth Newton methods. In our case, we consider a Tikhonov term in the $L^2$--norm. Although existence of optimal controls can be argued in this case under certain conditions, uniqueness of the solution is not expected as shown in a simple example below.  

Due to the lack of convexity and differentiability these costs are difficult to tackle numerically. In this paper, we address the numerical solution of this type of problems by regularizing the fractional $L^q$ quasinorms; for this purpose, we introduce a Huber--like smoothing function which regularizes the nonconvex $L^q$ term. In this way, we obtain a family of regularized nonsmooth problems whose objective functional can be expressed as a DC-function ( ``DC" stands for difference of convex functions), which reveals the underlying convexity of this class of problems. Although the regularized problem remains nonconvex and nondifferentiable, we can take advantage of the DC structure of the functional by applying known tools from convex analysis  and DC programming theories in order to derive optimality conditions and prove that the regularization is consistent. Moreover, we propose a numerical method based on the \emph{DC-Algorithm (DCA)}. It follows that the proposed DC splitting leads to a primal--dual updating that  only requires the numerical resolution of a convex $L^1$--norm penalized optimal control problem in each iteration, for which there are efficient numerical methods at hand. 

It is worth to mention that although our methodology is proposed for elliptic problems, it can be extended for different boundary conditions, parabolic problems or  optimal control problems involving other type of equations.

This paper is organized as follows. In Section 1 we introduce the non convex optimal control problems endowed with $L^q$--functionals with $q=\frac1p$, and $p>1$. In Section 2 we propose a Huber--like smoothing function in order to regularize the nonconvex optimal control problems. We show that the regularized problems can be expressed as a difference of convex functions and derive optimality conditions in Section 3. The box--constrained case is discussed at the end of this section. In addition, we provide a proof that the solution of the regularized version of the optimal control problem approximates the solution of the original one when the regularizing parameter tends to infinity. Section 4 is devoted to the numerical solution by proposing a DC--Algorithm based method. We finish this article by showing numerical examples and numerical evidence of the efficiency of the proposed method.\\

\subsection{Setting of the problem}
For $p>1$, let us define the mapping $\fsparse_p : L^2(\om) \arrow \reals $ by
\begin{equation}\label{eq:fractional1}
	u \mapsto \fsparse_p (u) := \int_\om |u|^{\frac{1}{p}}.
\end{equation}

Let $\om$ a bounded Lipschitz domain in $\reals ^n$  ($n=2$ or $n=3$) with boundary $\Gamma$. We are interested in the following optimal control problem  involving a penalization term of the form \eqref{eq:fractional1}. For $\alpha>0$ and $\beta>0$ we consider the optimal control problem:
\begin{equation}
\tag{$P$} \label{e:OCP}
\begin{cases}
\displaystyle\min_{(y,u)\in  H_0^1(\om) \times L^2(\om)} ~\frac{1}{2}\| y-y_d \|^2_{L^2(\om)}+\frac{\alpha}{2}\|u\|^2_{L^2(\om)}+\beta \fsparse_p(u)\\
\hbox{ subject to: }\\
\hspace{40pt}\begin{array}{rll}
A y=&u + f, &\hbox{in  } \om, \\
y=&0, &\hbox{on  }  \Gamma,
\end{array}
\end{cases}
\end{equation}
where $f$ is a given function in $L^2(\om)$ and $A$ is a uniformly elliptic second order differential operator of the form:
\begin{equation}\label{eq:A}
(Ay) (x) = - \sum_{i,j=1}^n \frac{\partial}{\partial x_i}	\l( a_{ij}(x)\frac{\partial y(x)} {\partial x_j }\r) +c_0 y(x).
\end{equation}
Here, the coefficients $a_{ij} \in C^{0,1}(\bar \om)$, and $c_0 \in L^{\infty}(\om)$. Moreover, the matrix $(a_{ij})$ is symmetric and fulfill the uniform ellipticity condition: \[ \exists\, \sigma>0 : \quad \ds  \sum_{i,j=1}^n a_{ij}(x) \xi_i\xi_j \geq \sigma |\xi|^2, \quad \forall \xi \in \reals^n, \text{for almost all } x \in \om. \]
We will denote the adjoint  of $A$ by $A^*$. Moreover, associated to the elliptic operator $A$, we define the bilinear form 
$$ a(y,v):= \int_{\om} \sum_{i,j=1}^n a_{ij}(x)\frac{\partial y(x)} {\partial x_j }\frac{\partial v(x)} {\partial x_j }  +c_0 y(x)v(x) \, dx, $$  which we use to define the associated variational problem problem:
\begin{equation}\label{eq:elliptic}
	a(y,v) = (w,v)_{L^2(\om)}, \quad \forall v\in H_0^{1}(\om).
\end{equation}
It is well known that \eqref{eq:elliptic} has a unique solution belonging to the space $H^1_0(\om)$. Let  $S:L^2(\om ) \arrow H_0^1(\om)$ be the linear and continuous operator which assigns to every $w \in L^2(\om)$ the corresponding solution $y=y(w) \in H_0^1(\om)$ satisfying \eqref{eq:elliptic}. Thus, the state equation: $Ay=u$ in $\om$,  with homogeneous Dirichlet boundary conditions, considered in \eqref{e:OCP}, is understood in the weak sense c.f. \eqref{eq:elliptic}. In this way, the state $y$ associated to the control $u$ has the representation  $y=S(u + f)$, which in turn allows us to formulate the usual reduced optimization problem: 
\begin{equation}
\tag{$P'$} \label{e:OCP'}
\displaystyle\min_{u \in L^2(\om)} J(u):=~\frac{1}{2}\| Su + Sf-y_d \|^2_{L^2(\om)}+\frac{\alpha}{2}\|u\|^2_{L^2(\om)}+\beta \fsparse_p(u).
\end{equation}

\begin{theorem} \label{t:existP}
There exists a solution $\bar u \in L^2(\om)$ for the reduced problem  \eqref{e:OCP'}.
\end{theorem}

We postpone the proof of this result to Section \ref{s:DCA}, where we proove that a sequence of solutions of approximating problems of the form $\min_{u} J_{\gamma}(u)$, converges to the solution of \eqref{e:OCP'}.

\begin{remark}
	The question of uniqueness is more delicate. The following example of the minimization of a real function has two solutions.
	Let $f : \reals \arrow \reals$ given by $f(x) = \frac12 (x - a)^2 + \beta |x|^{\frac12}$. By choosing $a=1+\frac12$ and $\beta=1$, it is easy to verify that $f$ has two minimum points at $x_{1} = 0$ and $x_2=1$ with the minimum value $f(0)=f(1)=\frac98$. Therefore, we cannot expect uniqueness of the solution for problem \eqref{e:OCP'} in view of the nonconvexity of cost function. 
	  
\end{remark}
 
Following the work of Stadler \cite{stadler09}, where $L^1$--norm penalization optimal control problems are considered, we expect that some analogous properties also hold for problem \eqref{e:OCP}. For example, it is expected that a local solution for \eqref{e:OCP} vanishes if the parameter $\beta$ is large enough. We address this question in the following lemma.
\begin{lemma}\label{l:null_sol} Let $S^*$ be the adjoint operator of $S$, and let $M>0$.
 If $\beta \geq \beta_0$ with $\beta_0= M^{\frac{p-1}{p}}\, \norm{S^{*}(Sf - y_d)}_{L^\infty(\om)}$ , then problem \eqref{e:OCP} has a local minimum at $\bar u = 0$ in $B_\infty(0, M)$ (the unit open ball in $L^\infty(\om)$)  with associated state $ y_0 := S f $.	
\end{lemma}
\begin{proof} Taking into account the reduced form \eqref{e:OCP'}, we argue analogously to \cite[Lemma 3.1]{stadler09}. Let us take $u \in B_\infty(0, M)$, then $|u(x)| < M$ for almost all $x$ in $\om$. Computing the difference of the cost values we have:
\begin{align*}
	J(u) - J(0) = &\frac{1}{2}\| Su + Sf-y_d \|^2_{L^2(\om)}+\frac{\alpha}{2}\|u\|^2_{L^2(\om)}+\beta \fsparse_p(u) \\ 
	& - \frac{1}{2}\| Sf-y_d \|^2_{L^2(\om)} \\
	=& \frac12 \| S u \|^2_{L^2(\om)} + (Su,Sf-y_d)_{L^2(\om)} + \frac{\alpha}{2}\|u\|^2_{L^2(\om)}+\beta \fsparse_p(u) \\
	\geq & \frac12 \| S u \|^2_{L^2(\om)} - \norm{u}_{L^1(\om)} \norm{S^*(Sf-y_d)}_{L^\infty(\om)} + \frac{\alpha}{2}\|u\|^2_{L^2(\om)}+\beta \fsparse_p(u), \\
	 \geq & \int_{\om} \beta |u|^{\frac1p} - |u| \norm{S^*(Sf-y_d)}_{L^\infty(\om)}\, dx \\
	 \geq & \int_{\om} \beta_0 |u|^{\frac1p} - |u| \norm{S^*(Sf-y_d)}_{L^\infty(\om)}\, dx.
\end{align*}
By the definition of $\beta_0$ it follows that 
\begin{align*}
	J(u) - J(0)	 \geq & \int_{\om} \l( M^{\frac{p-1}{p}}-|u|^{\frac{p-1}{p}} \r)|u|^{\frac1p} \norm{S^*(Sf-y_d)}_{L^\infty(\om)}\, dx > 0,
\end{align*}
where the nonnegativity is obtained by our assumption $u \in B_\infty(0,M)$.
\end{proof}

%
%
%

\section{The Regularized Optimal Control Problem}

\subsection{Huber--type regularization}
In order to analyze problem \eqref{e:OCP} we formulate a family of regularized problems, by means of the following Huber--type regularization of the absolute value.  Extending the classical Huber $C^1$ regularization of the absolute value, we propose a Huber regularization  $\fsparse_{p,\gamma}$ which takes into account the fractional powers defining $\fsparse_{p}$. The resulting function to the power $1/p$ is a locally convex regularization for the nonconvex and non differentiable term, see Figure \ref{fig:huber_approx} below. For $\gamma \gg 1$, we define

\begin{equation}\label{eq:huber}
\ds
{h}_{p,\gamma} (v)=
\left\{ \begin{array}{ll}
\frac{\gamma^{p-1}}{p}|v|^p,& \hbox{if } v \in[-\frac{1}{\gamma},\frac{1}{\gamma}],\vspace{2mm}\\
|v| + \frac{1}{\gamma}\frac{1-p}{p},&   \text{ otherwise. }
\end{array} \right.
\end{equation}

 \begin{figure}[ht!]
\centering
\begin{subfigure}{\textwidth}
\includegraphics{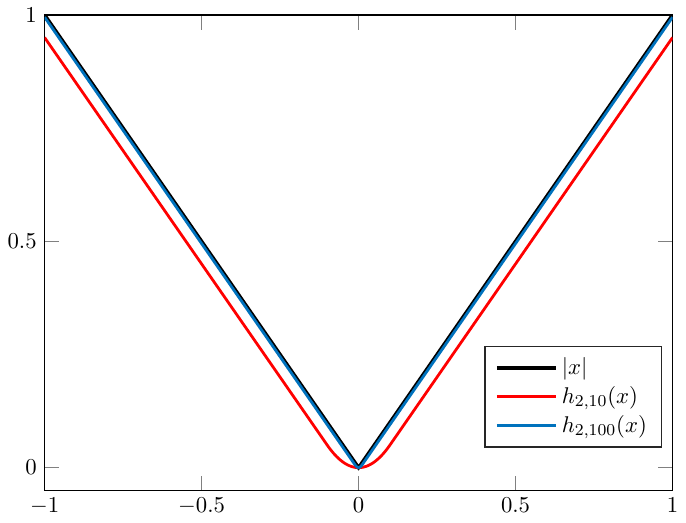}
\hfill
\includegraphics{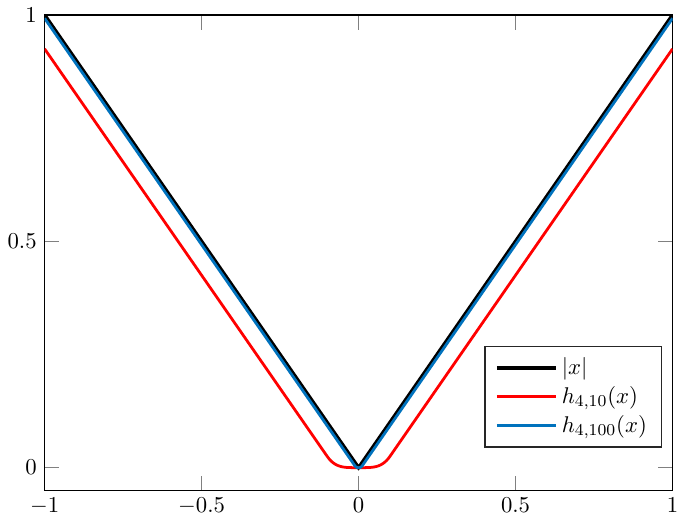}
\end{subfigure}

\begin{subfigure}{\textwidth}
\includegraphics{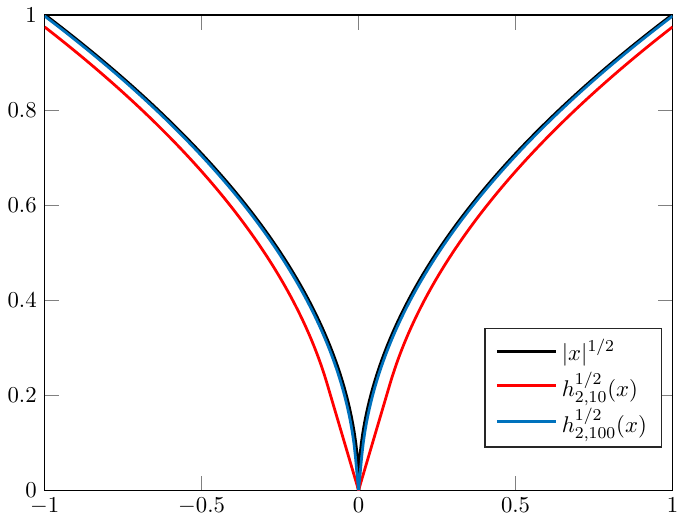}
\hfill
\includegraphics{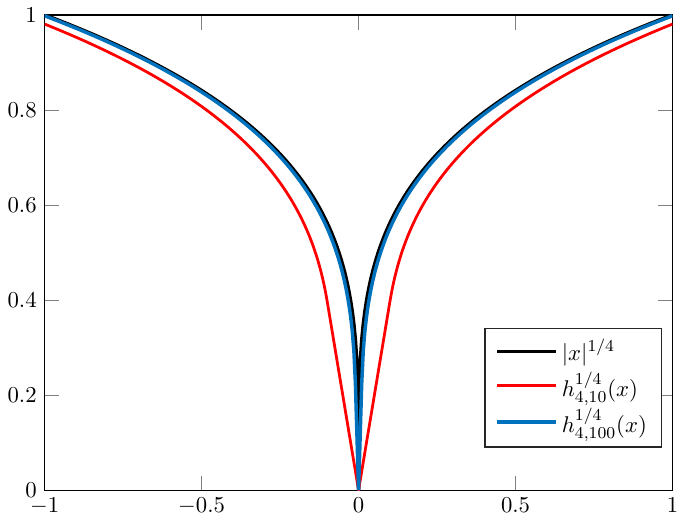}
\end{subfigure}
\caption{Exact (black) and regularized penalizations for the absolute value (first row) and the function $|x|^{1/p}$ in the second row, for parameters $\gamma=10$ (red) and $\gamma=100$ (blue), for $p=2$ (left) and $p=4$ (right).}
\label{fig:huber_approx}
 \end{figure}

\begin{remark}\label{r:huber_reg}
	The function $h_{p,\gamma}$ is a local regularization of the absolute value for different smoothing polynomial powers. In addition, notice that by construction, we have the relation 
	\begin{equation}\label{eq:hub1}
		h_{p,\gamma}(v) \leq |v|, \quad \forall v\in \reals.
	\end{equation}
It is worth to notice that \eqref{eq:huber} is different from the local regularization proposed in \cite{itoku2014}[pg. 1971 eq.(5.1)] which majorizes $\fsparse_p(u)$. Both regularization terms can be used to compute upper and lower bounds for the cost functions of \eqref{e:OCP}, respectively. Although they may appear similar, observe that \eqref{eq:huber} approximates $g$ nonsmoothly in a neighborhood of 0. This fact is crucial to express our objective functional as a difference of convex functions. In fact, the representation as a DC--function is not possible using the regularization proposed by \cite{itoku2014}.
Therefore, by using the Huber--type regularization we are able to appproximate \eqref{e:OCP} by sequence of $L^1$--sparse problems. The resulting DC--algorithm will be introduced in Section \ref{s:DCA}. 
\end{remark}

Now, we have the basic tool in order to formulate a regularized version of $\eqref{e:OCP}$. We introduce the function $\fsparse_{p,\gamma} $ defined by
\begin{equation}\label{eq:fractional}
	u \mapsto \fsparse_{p,\gamma} (u) := \int_\om h_{p,\gamma}(u(x))^{\frac{1}{p}} dx.
\end{equation}
The regularized problem is obtained by replacing $\Upsilon_p$  by $\fsparse_{p,\gamma} $. Therefore, the surrogate problem parameterized by  $\gamma$ reads:
\begin{equation}
\tag{$P_\gamma$} \label{e:OCP2}
\begin{cases}
\displaystyle\min_{(y,u)} ~\frac{1}{2}\| y-y_d \|^2_{L^2(\om)}+\frac{\alpha}{2}\|u\|^2_{L^2(\om)}+\beta \fsparse_{p,\gamma}(u)\\
\hbox{ subject to:}\\
\hspace{40pt}\begin{array}{cl}
Ay=u +f &\hbox{in  } \om, \\
y=0 &\hbox{on  }  \Gamma.
\end{array}
\end{cases}
\end{equation}

We proceed to formulate the reduced optimal control problem from \eqref{e:OCP2} by replacing the control--to--state operator $S$. Let $F$ be the regular part of the functional, which is $F(u) = \frac{1}{2}\| Su-y_d \|^2_{L^2(\om)}+\frac{\alpha}{2}\|u\|^2_{L^2(\om)}$. Thus, we have the reduced problem,

\begin{equation}
\label{e:OPT1}
\displaystyle\min_{u} J_\gamma(u):=~F(u) +\beta \fsparse_{p,\gamma}.
\end{equation}
From \cite[Lemma 5.1]{itoku2014} it is known that if a sequence $(u_n)_{n\in \mathbb{N}}$ is such that $u_n \arrow u$ in $L^1(\om)$ then $\fsparse_{p} ( u_ n) \arrow \fsparse_{p} ( u )$ as $n \arrow \infty$. In the case of $\fsparse_{\gamma,p}$ we have the following continuity property.

\begin{lemma}\label{l:conv1}
Let $(u_n)$	be a sequence such that $u_n \arrow u$ in $L^1(\om)$. Then $$ \fsparse_{p,\gamma} (u_n) \arrow  \fsparse_{p,\gamma} (u), \quad\text{when } n \arrow \infty,$$
for all $p>1$ and all $\gamma>0$.
\end{lemma}

\begin{proof}
Analogously to \cite[Lemma 5]{itoku2014} we define the following sets:
\begin{equation}\nonumber
\begin{array}{ll}
	\om_{n,1} = \{x: |u(x)|\leq\frac1\gamma  \text{ and } |u_n(x)|\leq\frac1\gamma  \}, \\
	\om_{n,2} = \{x: |u(x)|> \frac1\gamma \text{ and } |u_n(x)|>\frac{1}{\gamma} \}, \\
	\om_{n,3} = \{x: |u(x)| \leq  \frac1\gamma \text{ and } |u_n(x)| >\frac{1}{\gamma} \} \cup \{x: |u(x)| > \frac1\gamma \text{ and } |u_n(x)| \leq \frac{1}{\gamma} \},
\end{array}
\end{equation}
which we use to estimate $\l|\int_{\om} h_{p,\gamma}( u(x) )^{\frac1p} - h_{p,\gamma} (u_n(x))^{\frac1p}\,dx \r|$ according to \eqref{eq:huber}. 
Therefore, in $\om_{n,1}$ we have that
\begin{align}\label{eq:om_est0}
\l| \int_{\om_{n,1}} h_{p,\gamma}( u(x) )^{\frac1p} - h_{p,\gamma} (u_n(x))^{\frac1p}\,dx \r|	&\leq \l(\frac{\gamma^{p-1}}{p}  \r)^{\frac1p} \int_{\om_{n,1}}\l| \,|u(x)|- |u_n(x)|\, \r|\,dx, \nonumber \\
& \leq \l(\frac{\gamma^{p-1}}{p}  \r)^{\frac1p} \int_{\om} |u(x) -u_n(x)|\,dx \arrow 0.
\end{align}

Now, in $\om_{n,2}$ we can estimate
\begin{align}
\l| \int_{\om_{n,2}} h_{p,\gamma}( u(x) )^{\frac1p} - h_{p,\gamma} (u_n(x))^{\frac1p}\,dx \r|	&\leq  \int_{\om_{n,2}}\l| \l(|u(x)| + \frac1\gamma \frac{1-p}{p}\r)^{\frac1p}- \l(|u_n(x)| + \frac1\gamma \frac{1-p}{p}\r)^{\frac1p}\, \r|\,dx \nonumber \\
&\leq \int_{\om_{n,2}} \l| \, |u(x)| - |u_n(x)|\,\r|^{\frac1p}\, dx, \nonumber \\
&\leq \int_{\om_{n,2}} \l| \, u(x) - u_n(x)\,\r|^{\frac1p}\, dx. \nonumber
\end{align}
By applying H\"older inequality in the last integral, and by our convergence assumption we have
\begin{align}\label{eq:om_est1}
\l| \int_{\om_{n,2}} h_{p,\gamma}( u(x) )^{\frac1p} - h_{p,\gamma} (u_n(x))^{\frac1p}\,dx \r|	&\leq  |\om|^{^{\frac{p}{p-1}}}\l(\int_{\om} |u(x) -u_n(x)|\,dx \r)^{\frac1p} \nonumber \\
& \quad \arrow 0. \,
\end{align}
Finally, we estimate in $\om_{n,3}$. Without loss of generality we assume that $\{x: |u(x)| \leq \frac1\gamma \text{ and } |u_n(x)| >\frac{1}{\gamma} \}$. The neglected part can be argued in the same way by interchanging the role of $|u(x)|$ and $|u_n(x)|$.
Taking into account that the relation: $|u(x)| \leq 1/\gamma < |u_n(x)|$ is fulfilled in $\om_{n,3}$, it follows that

$$\l(\frac{\gamma^{p-1}}{p}  \r)|u(x)|^p < |u_n(x)| + \frac1\gamma \frac{1-p}{p}, $$
which implies
\begin{align}\label{eq:om_est2}
\l| \int_{\om_{n,3}} h_{p,\gamma}( u(x) )^{\frac1p} - h_{p,\gamma} (u_n(x))^{\frac1p}\,dx \r| & \leq  \int_{\om_{n,3}} \l| h_{p,\gamma}( u(x) )- h_{p,\gamma} (u_n(x)) \r|^{\frac1p}  \, dx \nonumber\\
&= \int_{\om_{n,3}}\l| \l(\frac{\gamma^{p-1}}{p}  \r) |u(x)|^p-  |u_n(x)| - \frac1\gamma \frac{1-p}{p} \, \r|^{\frac1p} \,dx \nonumber \\
& =  \int_{\om_{n,3}} \l( |u_n(x)| + \frac1\gamma \frac{1-p}{p} - \l(\frac{\gamma^{p-1}}{p}  \r)  |u(x)|^p \, \r)^\frac1p\,dx.
\end{align}
Furthermore, in $\om_{n,3}$ we have that $ \frac{1}{\gamma p} <|u_n(x)| + \frac1\gamma \frac{1-p}{p} < |u_n(x)|$, from which we obtain that
\begin{align}\label{eq:om_est3}
	|u_n(x)| + \frac1\gamma \frac{1-p}{p} <  |u_n(x)|^p \l(\frac{\gamma^{p-1}}{p} \r).
\end{align}
By replacing \eqref{eq:om_est3} in \eqref{eq:om_est2} we get the following relation

\begin{align}\label{eq:om_est4}
\l| \int_{\om_{n,3}} h_{p,\gamma}( u(x) )^{\frac1p} - h_{p,\gamma} (u_n(x))^{\frac1p}\,dx \r|  
& \leq \l(\frac{\gamma^{p-1}}{p} \r)^\frac{1}{p} \int_{\om_{n,3}} \l( |u_n(x)|^p  -   |u(x)|^p \, \r)^\frac1p\,dx ,\nonumber\\
& = \l(\frac{\gamma^{p-1}}{p} \r)^\frac{1}{p} \int_{\om_{n,3}} \l|\, |u_n(x)|^p  -   |u(x)|^p \, \r|^\frac1p\,dx ,\nonumber \\
& \leq \l(\frac{\gamma^{p-1}}{p} \r)^\frac{1}{p} \int_{\om} \l|\, |u_n(x)|^p  -   |u(x)|^p \, \r|^\frac1p\,dx.
\end{align}
By applying the mean value theorem, there is a $\xi(x)$ such that  $|u(x)|< \xi(x) < |u_n(x)| $ for almost all $x$ in $\om_{n,3}$ that satisfies $|u_n(x)|^p  -   |u(x)|^{p} = p |\xi (x)|^{p-1} (|u_n(x)|- |u(x)|)$. Hence, using this relation and applying H\"older inequality we have
\begin{align}
\int_{\om_{n,3}} \l|\, |u_n(x)|^p  -   |u(x)|^p \, \r|^\frac1p\,dx \nonumber
& \leq \int_{\om_{n,3}} \, p^{\frac1p}|\xi(x)|^{\frac{p-1}{p}} \l| |u_n(x)|  -   |u(x)| \r|^\frac1p \,dx.\nonumber \\
&  \leq p^{\frac1p}\int_{\om_{n,3}} \, |\xi(x)| \,dx  \int_{\om_{n,3}}\l| \, |u_n(x)|  -   |u(x)|\,\r|\,dx.\nonumber 
\end{align}

Thereby, the right--hand side of \eqref{eq:om_est4} tends to 0 as $n\arrow 0$.  Finally, collecting estimates \eqref{eq:om_est0}, \eqref{eq:om_est1} and \eqref{eq:om_est4} the result of the lemma is proved. 
\end{proof}

\begin{lemma}\label{l:uniconv_nonconvex}
$J_\gamma(u)$ converges to $J(u)$ uniformly as $\gamma \arrow \infty$, for any $u\in L^2(\om)$.
\end{lemma}
\begin{proof}
 We argue the uniform convergence of $J_{\gamma}$ to $J$ by using the definition of the Huber regularization \eqref{eq:huber}. Since $J_{\gamma}$ and $J$ differ on the nonconvex term, we analyze the difference $| \fsparse_{p,\gamma} (u) -\fsparse_{p}(u) |$ in the sets $\om_\gamma = \{ x \in \om: |u(x)| \leq \frac{1}{\gamma} \}$ and $\om^c_\gamma = \{ x \in \om: |u(x)| >\frac{1}{\gamma} \}$ as follows: 
\begin{align}
\l| \int_{\om} h_{p, \gamma} (u)^\frac{1}{p} - |u|^{\frac{1}{p}}\, dx \r|	&\leq \int_{\om_{\gamma}}\l | h_{p, \gamma} (u)^\frac{1}{p} - |u|^{\frac{1}{p}} \r|\, dx  + \int_{\om^c_{\gamma}} \l | h_{p, \gamma} (u)^\frac{1}{p} - |u|^{\frac{1}{p}} \r| \, dx \nonumber \\
& \leq  \int_{\om_{\gamma}} \l| \frac{\gamma^{\frac{p-1}{p}}}{p^\frac1p}|u| - |u|^{\frac1p}  \r|\, dx  + \int_{\om^c_{\gamma}} \l | \l(|u| + \frac{1}{\gamma}\frac{1-p}{p}\r)^\frac{1}{p} - |u|^{\frac{1}{p}} \r| \, dx \nonumber
\end{align}
Using the fact that $|u(x )|\leq \frac{1}{\gamma}$ in $\om_\gamma$ we have
\begin{align}
\l| \int_{\om} h_{p, \gamma} (u)^\frac{1}{p} - |u|^{\frac{1}{p}}\, dx \r|	&\leq  \int_{\om_{\gamma}}  \frac{1}{ \gamma^{\frac{1}{p}} p^\frac1p} + {\frac{1}{\gamma^{\frac1p}}}  \, dx  + \frac{1}{\gamma^\frac{1}{p}}\int_{\om^c_{\gamma}}   \l|\frac{1-p}{p}\r|^\frac{1}{p} \, dx ,\nonumber \\
%
&\leq  \int_{\om }  \frac{1}{ \gamma^{\frac{1}{p}} p^\frac1p} + {\frac{1}{\gamma^{\frac1p}}}  \, dx + \frac{1}{\gamma^\frac{1}{p}}\int_{\om}   \l|\frac{1-p}{p}\r|^\frac{1}{p} \, dx , \nonumber 
\end{align}
where the last terms clearly tends to 0 as $\gamma \arrow \infty$. 
\end{proof}

%
%

 \section{Existence and Optimality Conditions for the regularized problem}
 
Our aim in this section is deriving an optimality system for problem \eqref{e:OCP2} via a DC--programming approach. As mentioned earlier, the key idea is introducing an $L^1$--norm penalization which allows us to formulate our problem as a minimization of a difference of convex functions, with functions $G$ and $H$ such that:
\begin{equation}\label{eq:dc}
J_\gamma(u) = G(u) - H(u).	 
\end{equation}
A function that can be expressed in this form is known as a DC--function and several problems involving this type of functions have been analyzed, see the monograph of Hiriart Urruty \cite{hiriart1989} or in \cite{dinh2014}. 

Let us focus on how to express the cost function of problem \eqref{e:OCP2} as a convenient \emph{difference of convex functions} and then rely on the theory of DC programming. We start by introducing the following quantity, which will be frequently used throughout this paper:

\begin{equation}
\delta_\gamma = \frac{\gamma^{\frac{p-1}{p}}}{p^{\frac1p}}.	
\end{equation}
The next step is to define $G$ and $H$ in \eqref{eq:dc} as follows: 

\begin{align}\label{eq:GH}
\begin{array}{ll}
&\begin{array}{lrlll}
	G: & L^2(\om) & \arrow &\reals \\
	   & u		& \mapsto &  G(u) &: = \frac12 \norm{Su +Sf-y_d}^2_{L^2(\om)}	 +\alpha \norm{u}^2_{L^2(\om)} + \beta \delta_\gamma \norm{u}_{L^1(\om)} \\
	   & & &  &= F(u) + \beta \delta_\gamma \norm{u}_{L^1(\om)},
\end{array} \\
&\begin{array}{lrll}
	H: & L^2(\om) & \arrow &\reals \\
	   & u		& \mapsto &  H(u) : =  \beta \l( \delta_\gamma \norm{u}_{L^1(\om)} - \fsparse_{p,\gamma}(u) \r).
\end{array}
\end{array}
\end{align}

\begin{lemma}\label{l:j} The real function $j : \reals \arrow \reals_+\cup \{0\}$, defined by
\begin{equation}
j(z) = \l\{ 
\begin{array}{ll}
\delta_\gamma |z| - \l(|z| + \frac{1}{\gamma} \frac{1-p}{p} \r)^{\frac{1}{p}}, & \text{ if } |z|  >  \frac{1}{\gamma} \\
0, & \text{ if } |z| \leq \frac{1}{\gamma} ,
\end{array}
\r.
\end{equation}
is nonnegative, convex  and continuously differentiable and its derivative in $z \in \reals$, is given by
\begin{equation}\label{eq:dj}
j'(z) = \l\{ 
\begin{array}{ll}
\delta_\gamma \, \sign(z) - \frac{1}{p}\l(|z| + \frac{1}{\gamma} \frac{1-p}{p} \r)^{\frac{1-p}{p}}\sign(z), & \text{ if } |z|  >  \frac{1}{\gamma} \\
0, & \text{ if } |z| \leq \frac{1}{\gamma} .
\end{array}
\r.
\end{equation}

\end{lemma}
\begin{proof}
%
Let us first check differentiability. It is clear that $j$ is differentiable if $|z| < \frac{1}{\gamma}$ or $|z| > \frac{1}{\gamma}$, where $j'(z) = 0$ and $j'(z) = \delta_\gamma \sign(z)- \frac{1}{p}(|z| + \frac{1}{\gamma}\frac{1-p}{p})^{\frac{1-p}{p}} \sign(z)$, respectively. Therefore, we check differentiability at $z = \pm\frac{1}{\gamma}$. Consider $z=-\frac{1}{\gamma}$, since $j(\pm\frac{1}{\gamma}) = 0$ and \\ \(|-\frac{1}{\gamma}+h| < \frac{1}{\gamma}\) for sufficiently small $h$, we have that
$ \ds	\lim_{h \arrow 0^+} \frac{j(z + h) - j(z)}{h}  = \lim_{h \arrow 0^+} \frac{j(-\frac{1}{\gamma} + h)}{h} \allowbreak = 0$.
On the other hand, since $-\frac{1}{\gamma} + h <0$ for sufficiently small $h$
\begin{align*}
\ds	\lim_{h \arrow 0^-} &\frac{j(z + h) - j(z)}{h}  = \lim_{h \arrow 0^-}\frac{j(-\frac{1}{\gamma} + h)}{h}  \\
&= \lim_{h \arrow 0^-}\frac{\delta_\gamma \l( \frac{1}{\gamma} - h \r) - \l( \frac{1}{\gamma}-h + \frac{1}{\gamma}\frac{1-p}{p} \r)^{\frac{1}{p}}}{h} =\lim_{h \arrow 0^-}\frac{ \l( \frac{1}{\gamma p} \r)^{\frac{1}{p}} - \delta_\gamma h  - \l( \frac{1}{\gamma p}-h  \r)^{\frac{1}{p}}}{h} 
,
\end{align*}
where we apply the binomial theorem to get
\begin{align*}
\lim_{h \arrow 0^-}\frac{ \l( \frac{1}{\gamma p} \r)^{\frac{1}{p}} - \delta_\gamma h  - \l( \frac{1}{\gamma p}-h  \r)^{\frac{1}{p}}}{h} &= 
\lim_{h \arrow 0^-}\frac{ \l( \frac{1}{\gamma p} \r)^{\frac{1}{p}} - \delta_\gamma h  - \l( \frac{1}{\gamma p} \r)^{\frac{1}{p}} -\frac{1}{p} \l( \frac{1}{\gamma p} \r)^{\frac{1-p}{p}}h +o(h) }{h} \\
&=\lim_{h \arrow 0^-} \frac{o(h)}{h} =0.
\end{align*}
Therefore $j'(-\frac{1}{\gamma})=0$. Analogously, it also follows that $j'(\frac{1}{\gamma})=0$, which implies formula \eqref{eq:dj}. Moreover, a straightforward observation reveals that $j'$ is continuous, therefore $j$ is continuously differentiable. Convexity follows by noticing that  the function  $ \ds\reals_+ \ni z \mapsto (z + \frac{1}{\gamma}\frac{1-p}{p} )^{1/p}$ is concave, because it is the composition of an affine function and a concave function. Thus, for $z >\frac1\gamma,$ we find that the function 
 \[  \ds\reals_+ \ni z \mapsto  \delta_\gamma z - \l(z + \frac{1}{\gamma}\frac{1-p}{p} \r)^{1/p} \]
is convex and monotonically increasing, which, by composition with the absolute value, implies the convexity of $j$. Finally, we make the simple but important observation that $j$ vanishes in the interval $[-\frac{1}{\gamma}, \frac{1}{\gamma}]$. This, together with the convexity of $j$, implies that $j$ is nonnegative.
\end{proof}
\linebreak

Now, by employing the function $j$ we can write $H$ as follows:

\begin{equation}\label{eq:Hj}
\begin{array}{lrll}
	H: & L^2(\om) & \arrow &\reals \\
	   & u		& \mapsto &  H(u)  = \ds\int_{\om} j(u) dx .
\end{array}	
\end{equation}

\begin{lemma}\label{l:convexity}
	The functions $G$ and $H$ defined in \eqref{eq:GH} are convex.
\end{lemma}

\begin{proof}
Since $\alpha \geq 0$ and $\beta \geq 0$, it is clear that function $G$ is strictly convex if $\alpha + \beta  >0$. In the case of $H$, convexity  directly follows from Lemma \ref{l:j}.
\end{proof}
\linebreak

Having defined the functions $H$ and $G$, it is clear that the representation \eqref{eq:dc} of $J_\gamma$  has been set up. Therefore, $J_\gamma$ is a DC-function and we can express optimality conditions in terms of $G$ and $H$ by considering the following formulation for problem \eqref{e:OPT1}:

\begin{equation}
\label{e:OPT-DC}\tag{DC}
\displaystyle\min_{u} J_\gamma(u)= G(u) - H(u),
\end{equation} 
\begin{lemma}\label{l:diffH}
	The function $H$ defined in \eqref{eq:GH} is G\^ateaux differentiable, and its derivative $H_G'( u;\cdot)$ is represented by $( \beta  w , \cdot )$, where $ w \in L^2(\om)$ depends on $ u$, $p$ and $\gamma$, and it is given by
\begin{equation}\label{eq:dH}
		\ds
 w (x): =
\left\{ \begin{array}{ll}
\ds \l[ \delta_\gamma - \frac1p \l( | u(x)| + \frac{1}{\gamma}\frac{1-p}{ p } \r)^{\frac{1-p}{p}} \r]\sign ( u (x)),& \hbox{if } | u(x)| > \frac{1}{\gamma},\vspace{2mm}\\
0,&   \text{ otherwise. }
\end{array} \right.
	\end{equation} 
\end{lemma}
\begin{proof} First, notice that $j'(z)$, given by \eqref{eq:dj}, satisfies that 
\begin{equation}\label{eq:wbound}
0<|j'(z)|=\l|\delta_\gamma - \frac{1}{p} \l( |z| + \frac{1}{\gamma} \frac{1-p}{p}\r)^{\frac{1-p}{p}}\r|< \delta_\gamma, \quad \text{for } |z| > \frac{1}{\gamma}. 
\end{equation}
Therefore, by using \eqref{eq:wbound} and the properties of $j$ established in Lemma \ref{l:j}, we apply \cite[Theorem 2.7, pg. 19]{ambro95} in order to deduce that the superposition operator $u \mapsto j(u)$ is G\^ateaux differentiable from $L^2(\om)$ into $L^2(\om)$. In addition, its G\^ateaux derivative in the direction $v$ is given by $j'(u)v \in L^2(\om)$. Hence, Theorem 7.4-1 in \cite{ciarlet2013} allows us to compute the G\^ateaux derivative of $H$ at $\bar u$ in any direction $v \in L^2(\om)$ by
\begin{equation}
H'_G (u,v) = \int_\om j'( u(x))v dx = (\beta w, v),
\end{equation}
with $w$ given by   \eqref{eq:dH}.
\end{proof}
\begin{theorem} \label{t:exist}
Let $U_{ad}$ be the feasible control set, and assume that $U_{ad}:=\{u\in L^2(\om): \Delta u \in H^{-1}(\om), \exists v\in \bar B(0,M) \subset L^2(\om) \text{ such that } -\Delta u + \frac{1}{\varepsilon} u =  \frac{1}{\varepsilon} v\}$, for a fixed $\varepsilon>0$ for a positive constant $M$. There exists a solution $\bar u  \in L^2(\om)$ for the regularized problem  \eqref{e:OCP2}.
\end{theorem}

\begin{proof} 
Existence of a solution can be argued by standard techniques. Let us fix $\gamma>0$ and consider a minimizing sequence $(u_k)_{k\in\mathbb{N}}$ for $J_\gamma$, defined in   $U_{ad}$. By the definition of $U_{ad}$, there exists a bounded sequence $(v_k) \subset \bar B(0,M)$ that satisfy $-\Delta u_k + \frac{1}{\varepsilon}u_k =\frac{1}{\varepsilon}v_k$ and $y_k \in H_0^1(\om)$ associated to $u_k$. Therefore, we extract  (without renaming) a weakly convergent subsequence $(u_k)_{k\in \mathbb N}$ in $H_0^1(\om)$ having weak limit $\bar u \in U_{ad}$. Let us denote $\bar y=S\bar u$. Moreover, because the compact embedding $H_0^1(\om)\hookrightarrow L^2(\om)$ we have that $u_k \arrow \bar u$  and $y_k \arrow \bar y $ strongly in $L^2(\om)$.

Using Lemma \eqref{l:conv1} and the continuity of the remaining terms of $J_\gamma$ we can pass to the limit:
\begin{align}
	J_\gamma(\bar u)  
				      & =  \lim_{k\arrow \infty} \frac{1}{2} \norm{y_k -y_d}_{L^2(\om)} + \frac{\alpha}{2} \norm{u_k}_{L^2(\om)} + \fsparse_{p,\gamma}{u_k}
				       = \inf_{u \in U_{ad}} J_\gamma (u).
\end{align}

\end{proof}

\subsection{First--order necessary conditions}
The following part of this paper moves on describing the derivation of first order necessary optimality conditions for problem \eqref{e:OCP2}. The conditions for local and global optimality can be found in \cite[Proposition 3.1 and 3.2]{hiriart1989} or in \cite{florettli2001}. We will use the following well known result from DC--programming theory, which permits the characterization of local minima.

\begin{proposition} Let $G$ and $H$, the convex functions defined in \eqref{eq:GH}. If $\bar u$ is a local minimum of the  DC--function $J_\gamma = G - H$, then $\bar u$ satisfies the following critical point condition:
\begin{equation}\label{eq:dc_optimality}
	\partial H(\bar u) \subset \partial G( \bar u).
\end{equation}	
\end{proposition}

The next result establishes an optimality system with the help of the last proposition.
\begin{theorem}\label{t:fonc} Let $\bar u$ be a solution of \eqref{e:OCP2}, then there exist:  $\bar y = S \bar u$ in $H_0^1(\om)$, 
an adjoint state $\phi \in H_0^1(\om)$, a multiplier $\zeta \in L^2(\om)$ and $\bar w$ given by \eqref{eq:dH} such that the following optimality system is satisfied: 
\begin{center}
\begin{subequations}
\label{eq:OPT_P'}
\begin{align}
& \begin{array}{rll}
A \bar y &= \bar u +f,  & \text{in } \om, \\
	   \bar y & = 0, & \text{on } \ga, 	
\end{array}\label{eq:state1}\\
& \begin{array}{rll}
A^* \bar \phi &= \bar y - y_d,  & \text{in } \om, \\
	   \bar \phi & = 0, & \text{on } \ga, 	
\end{array} \label{eq:adj1}\\
%
& \bar \phi +\alpha \bar u + \beta \, (\delta_\gamma \, \zeta -\bar w ) = 0,	\label{eq:gradeq}\\
& \begin{array}{lll}
\zeta(x) &=1, & \text{ if } \bar u (x) >0, \\
\zeta(x) &=-1, & \text{ if } \bar u (x) <0, \\
|\zeta(x)| &\leq 1,  & \text{ if } \bar u (x) =0, \\ 
\end{array} \label{eq:zeta} 
 \text{ for almost all } x \in \om. 
\end{align}
\end{subequations}	
\end{center}

\end{theorem}
\begin{proof} Clearly, equation \eqref{eq:state1} is equivalent to $S \bar u = \bar y$. By standard properties of subdifferential calculus, see for example \cite{hiriart2012}, the subdifferential of  $G$ at $\bar u$ is given by $\partial G(\bar u) = \nabla f(\bar u) + \beta\delta_\gamma\partial\, \norm{\cdot}_{L^1(\om)}(\bar u)   $.
By Lemmas \ref{l:conv1} and \ref{l:diffH}, it follows that $\partial H(\bar u)$ consists in the singleton $\{ \bar w \}$. Thus, condition \eqref{eq:dc_optimality} becomes 
\begin{align}\label{eq:fonc_1}
	\bar w & \in \nabla f (\bar u ) +\beta \delta_\gamma \partial\, \norm{\cdot}_{L^1(\om)}(\bar u).
\end{align}
Since $S$  is a linear and continuous operator from $L^2(\om)$ to $L^2(\om)$, the computation of $\nabla f(\bar u)$ is straightforward, see for instance \cite{delr2015}. Therefore, for $u\in L^2(\om)$ we have that 
\begin{align}
	\nabla f(\bar u)u & = (Su, S\bar u+Sf -y_d )_{L^2(\om)} + \alpha (u,\bar u)_{L^2(\om)}\nonumber \\
					  & = (u, \alpha \bar u + S^*(\bar y-y_d ))_{L^2(\om)}. \label{eq:fonc_2}
\end{align}
Moreover, by introducing the adjoint state $\bar \phi \in H_0^1(\om)$ as the solution of the adjoint equation: 
	\begin{equation}
	\begin{array}{rll}
A^* \bar \phi &= \bar y - y_d,  & \text{in } \om, \\
	   \bar \phi & = 0, & \text{on } \ga, 	
\end{array}	\nonumber
	\end{equation}
we are able to write $\bar \phi = S^*( \bar y -y_d)$ ($S^*$ denoting the adjoint control-to--state operator).

On the other hand, it is well known \cite[Chapter 0.3.2]{ioffe1979}, that any $\zeta \in \partial\, \norm{\cdot}_{L^1(\om)}(\bar u)$ is characterized by
\begin{equation}\label{eq:fonc_3}
\zeta (x) 
\l\{
	\begin{array}{lll}
=1, & \text{ if } \bar u (x) >0, \\
=-1, & \text{ if } \bar u (x) <0, \\
\in  [-1,1],  & \text{ if } \bar u (x) =0.
\end{array} 
\r.
\end{equation}
In this way, from \eqref{eq:fonc_2} we obtain that $\nabla f(\bar u)= \bar \phi + \alpha \bar u$ which together with \eqref{eq:fonc_3} imply the existence of $\zeta \in \partial\, \norm{\cdot}_{L^1(\om)}(\bar u) \subset L^\infty(\om)$ which allows us to write \eqref{eq:fonc_1} in the form:
\begin{align}\label{eq:fonc_4}
	 \,\bar \phi + \alpha \bar u + \beta (\delta_\gamma \zeta - \bar w ) = 0.
\end{align}
	
\end{proof}
{
\begin{corollary}
	 If $\bar u$ is a solution of \eqref{e:OCP2} with the associated quantities $\bar y$, $\bar \phi$, $\zeta$ and $\bar w$ satisfying \eqref{eq:OPT_P'} then, the following relations are fulfilled
	 \begin{subequations} \label{eq:u_phi}
	 \begin{align}
	 	\bar u(x) &= 0  \Leftrightarrow |\bar \phi| \leq \beta \delta_\gamma, \text{ and } \label{eq:u_phi.1}\\
	 	\zeta &= \mathcal{P}_{[-1,1]} \l[ -\frac1{\beta\delta_\gamma}  \bar\phi \r]. \label{eq:u_phi.2}
	 \end{align}	
	 \end{subequations}
\end{corollary}
\begin{proof} Let $x\in \om$ be such that $\bar u(x)=0$. Taking into account  \eqref{eq:zeta} and  \eqref{eq:dH}, it follows that $\zeta(x) \in [-1,1]$ and that $\bar w(x) =0$. Then, the gradient equation \eqref{eq:gradeq} implies  $|\bar \phi (x)|  \leq \beta \delta_\gamma $. Reciprocally, let us suppose that $x\in \om$ is such that  $|\bar \phi (x)|  \leq \beta \delta_\gamma $ holds. Let us  assume first that $\bar u(x)>0$, then by Lemma \ref{l:diffH} it follows that $\bar w \leq 0$ and \eqref{eq:gradeq} implies that  $$\bar \phi (x)= -\alpha\bar u(x) - \delta_\gamma \beta  + \beta \bar w(x)<-\beta\delta_\gamma,$$
which is a contradiction. On the other hand,  by assuming $\bar u(x)<0$, an analogous chain of arguments also lead us to a similar contradiction. Hence, we conclude that $\bar u(x) =0$, which proves \eqref{eq:u_phi.1}.

The second relation follows from  \eqref{eq:zeta}, \eqref{eq:u_phi.1} and  \eqref{eq:gradeq}. Indeed, observe that:

\begin{equation}\label{eq:zeta-w}
		\ds
\beta\delta_\gamma \zeta(x) - \beta w (x) 
\left\{ \begin{array}{ll}
=\ds   \frac1p \l( | u(x)| + \frac{1}{\gamma}\frac{1-p}{ p } \r)^{\frac{1-p}{p}} \sign ( u (x)),& \hbox{if } | u(x)| > \frac{1}{\gamma},\vspace{2mm}\\
\in [-\beta\delta_\gamma,\beta\delta_\gamma],&   \text{ otherwise. }
\end{array} \right. \nonumber
	\end{equation} 
Therefore, we have that $|\beta\delta_\gamma \zeta(x) - \beta \bar w (x)| \leq \beta\delta_\gamma$. If we assume $\bar \phi(x) >{\beta \delta_\gamma}$ then $ \alpha \bar u(x) < - {\beta \delta_\gamma} -  \beta \delta_\gamma \zeta (x) + \beta \bar w (x)\leq 0$ which, in view of \eqref{eq:zeta}, implies that $\zeta (x) = -1 $. Similarly, if $\phi(x) <-{\beta \delta_\gamma}$ we have that $\zeta (x) = 1 $. This, together with \eqref{eq:zeta}, implies \eqref{eq:u_phi.2}.
%
\end{proof} 
\begin{remark}\label{r:ctrlsupp_prop}
	The relations given by \eqref{eq:u_phi} have two important consequences: 
	\begin{enumerate}[(i)]
		\item \eqref{eq:u_phi.1} proves that sparsity of the solution \eqref{e:OCP2} is characterized by the adjoint state, solution of \eqref{eq:adj1}. Moreover, since $$ \frac12\norm{\bar y - y_d}^2_{L^2(\om)} \leq J_\gamma(\bar u) \leq J_\gamma(0) =\frac12\norm{ y_0 - y_d}^2_{L^2(\om)},$$ then, there exist positive constants $c$ and $C$, such that 
	       \begin{align}
	       	 \norm{\bar \phi}_{L^\infty(\om)} \leq c \norm{\bar y - y_d}_{L^2(\om)}  \leq C \norm{y_0 - y_d}_{L^2(\om)}:=m. \label{eq:phi_bound}
	       \end{align}
	Therefore, since $\gamma$ is fixed, a value of $\beta$ can be chosen such that $ \beta\delta_\gamma \geq m $. Then the adjoint state fulfills $\norm{\bar \phi}_{L^\infty(\om)} \leq \beta\delta_\gamma$ and by \eqref{eq:u_phi.1} the optimal control $\bar u$ must be zero. This complements the result from Lemma \ref{l:null_sol} regarding the existence of a parameter $\beta$ which enforces a null solution. 
	\item We also observe that equation \eqref{eq:phi_bound} implies that $\bar\phi$ is uniformly bounded for any $\gamma>0$. In view of \eqref{eq:phi_bound},  we claim that the set $\om_\gamma=\{ x: 0 < |u(x)| \leq 1/\gamma\}$ has zero Lebesgue measure for $\gamma$ large enough. This fact follows from the gradient equation \eqref{eq:gradeq} and the definition of $\bar w$. Indeed, if we suppose that $|\om_\gamma|>0$ this would imply that 	 $$ |\bar \phi(x)| >\beta\delta_\gamma -\frac{1}{\gamma}, \quad \text{ a.e. in }\om_\gamma, $$
	with the right--hand side of the last relation growing to infinity as $\gamma \arrow \infty$ because $\delta_\gamma \arrow \infty$. This is in contradiction with \eqref{eq:phi_bound}. Thus, $|\om_\gamma|=0$.
	
	Furthermore,  it turns out that a $\rho>0$ exists such that the set $\om_\rho:=\{ x: \frac{1}{\gamma} < |u(x)| \leq \frac{1}{\gamma}+\rho\}$ has zero Lebesgue measure for $\gamma$ large enough. Indeed, assume that  $|\om_\rho|>0$ for all $\rho>0$. Let us take some $\gamma_0$ such that $\frac{1}{\gamma}<\rho$ for all $\gamma>\gamma_0$. By using the  gradient equation \eqref{eq:gradeq} and the definition of $\bar w$, we have that 
	 \begin{align}
	 	 \frac{1}{\gamma} + \frac{\beta}{p} \l(\rho +\frac1p\frac{1}{\gamma }\r)^{(1-p)/p}\leq \l|\alpha\bar  u(x) + \frac{\beta}{p}  \l( | \bar u(x)| + \frac{1}{\gamma}\frac{1-p}{ p } \r)^{\frac{1-p}{p}}\sign{(\bar u(x))}\r|  = |\bar \phi(x)|,\nonumber
	 \end{align} 
	almost everywhere in $\om_\gamma$. Again, a contradiction is obtained by noticing that the lower bound of the last relation is unbounded if such $\rho$ does not exist. Therefore, in view of  \eqref{eq:phi_bound}, a $\rho>0$ must exist such that $|\om_\rho|=0$. In other words, if $\bar u(x)\not=0$, then  $|\bar u(x)|>\rho>0$ holds almost everywhere for some $\rho >0$ and for $\gamma$ sufficiently large.
	   \end{enumerate}
\end{remark}
}

An important question regarding the regularized problem \eqref{e:OCP2} is about the convergence of the solutions of \eqref{e:OCP2} to a solution of the original problem \eqref{e:OCP} when $\gamma \arrow \infty$. We address this question in the following results. 
{
\begin{theorem}\label{t:opsys_u*}
	 There exists a sequence  $(\bar u_\gamma )_{\gamma>0}$ of solutions for problem \eqref{e:OCP2} weakly converging to $u^*$ in $L^2(\om)$. Moreover, there exist $\phi^* \in H_0^1(\om)$, $y^* \in H_0^1$ and $\xi^*\in L^2(\om)$ satsifying the system:
\begin{subequations}
\label{eq:OPT*}
\begin{align}
& \begin{array}{rll}
A y^* &=  u^* +f,  & \text{in } \om, \\
	   y^*& = 0, & \text{on } \ga, 	
\end{array}\label{eq:state*}\\
& \begin{array}{rll}
A \phi^* &=  y^* - y_d,  & \text{in } \om, \\
	   \bar \phi & = 0, & \text{on } \ga, 	
\end{array} \label{eq:adj*}\\
%
& \alpha  u^*(x) +  \phi^*(x)+ \xi^*(x) = 0,	\label{eq:gradec_u*}\\
& \;\text{ for almost all } x \in \om. \nonumber
\end{align}
\end{subequations}
In addition, there exists a $\rho >0$ such that  
\begin{equation}\label{eq:supp_u*}
|u^*(x)|\geq \rho, \quad \text{ for almost all } x \in \om \backslash \{x \in \om: u^*(x) =0 \}.	
\end{equation}
\end{theorem}
\begin{proof}
	Here, we make explicit the fact that the quantities associated to the solution  of \eqref{e:OCP2} given in the system \eqref{eq:OPT_P'} depend on the regularization parameter $\gamma$. Therefore, the optimal control will be denoted by $\bar u_\gamma$, and the its associated quantities satisfying the optimality system \eqref{eq:OPT_P'} will be denoted by $\bar \phi_\gamma$, $\bar y_\gamma$, $\zeta_\gamma$ and $\bar w_\gamma$ respectively. 

We begin by noticing that the sequence $(\bar u_\gamma)_{\gamma>0}$ is bounded in $L^2(\om)$. Indeed,  since $y_0=S f$,  the optimality of $\bar u_\gamma$ for \eqref{e:OCP2} results in
\begin{align}
	\frac{\alpha}{2}\norm{\bar u_\gamma}^2_{L^2(\om)}\leq J_\gamma(\bar u_\gamma) & \leq J_\gamma (0) = \frac12 \norm{y_0- y_d}^2_{L^2(\om)}, \nonumber
\end{align}
	 which implies the boundedness of $(\bar u_\gamma)_{\gamma>0}$ in $L^2(\om)$ for $\alpha>0$.

As usual, reflexivity of $L^2(\om)$ allows us to extract a weakly convergent subsequence, denoted again by  $(\bar u_\gamma)_{\gamma>0}$ with limit $u^* \in L^2(\om)$.  Furthermore, the sequence $(\xi_\gamma)_{\gamma>0}$, defined by $\xi_\gamma:= \beta \delta_\gamma \zeta_\gamma -\beta \bar w_\gamma$, is also bounded in $L^2(\om)$ in view of equation \eqref{eq:gradeq}. Let  $\xi^* \in L^2(\om)$ be the weak limit of $(\xi_\gamma)_{\gamma>0}$ (after extracting a convergent subsequence). We denote by $y^*$ and by $\phi^*$ the corresponding solutions of equations $Ay = f + u^*$ in $H_0^1(\om)$ and  $A\phi = y^*-y_d$ in $H_0^1(\om)$, respectively. Thus, we will refer to $y^*$ and by $\phi^*$  as  the state and the adjoint state in $H^1_0(\om)$ associated to $u^*$. Notice that by the compact embedding $H^1_0(\om) \hookrightarrow \hookrightarrow L^2(\om)$ we have strong convergence of $\bar y_\gamma \arrow y^*$ and  $\bar \phi_\gamma \arrow \phi^*$ in $L^2(\om)$.
  Therefore, taking $\gamma \arrow \infty$ in equation \eqref{eq:gradeq}, we have
\begin{equation*}
	\lim_{\gamma \arrow \infty} (\bar \phi_\gamma  +\alpha \bar u_\gamma  + \beta \, (\delta_\gamma \, \zeta_\gamma -\bar w_\gamma), v )  = ( \alpha u^* +\phi^* + \xi^*,v) = 0, \quad \forall v\in L^2(\om),
\end{equation*}
which proves \eqref{eq:OPT*}.

Let us verify property \eqref{eq:supp_u*}. 
First, notice that if $\om_\rho$ denotes the subset of $\om$ where $|\bar u_\gamma (x)|\leq \rho$ then, in view of Remark \ref{r:ctrlsupp_prop}, it follows that $\bar u_\gamma(x)=0$ a.e. in $\om_\rho$ for $\gamma$ sufficiently large. Then, by a weak lower semicontinuity argument we have that $\int_{\om_\rho} (u^*(x))^2 dx\leq \liminf_{\gamma \arrow 0} \int_{\om_\rho}(\bar u_\gamma(x))^2 dx =0$ which implies that $u^*(x) = 0$ a.e. in $\om_\rho$.

On the other hand, based again on Remark \ref{r:ctrlsupp_prop} (ii) we may chose $\gamma$ sufficiently large and a Lebesgue point $x\in \om$ of $\bar u_\gamma$, such that $\rho < \bar u_\gamma (x)$. Therefore, 
\begin{equation}\label{u*(x).1}
\rho < \bar u_\gamma(x) = \lim_{r\arrow 0}\frac{1}{|B(x,r)|} \int_{B(x,r)} \bar u_\gamma (y) dy. 	
\end{equation}
Moreover, since $\bar u_\gamma \rightharpoonup u^*$ in $L^2(\om)$ we also have that $  \int_{B(x,r)} \bar u_\gamma (y) dy \arrow   \int_{B(x,r)} u^* (y) dy$ as $\gamma \arrow \infty$, for all $r \arrow 0$. Thus,  by using \eqref{u*(x).1}, for $\varepsilon>0$ there exists $\gamma_0$ such that for $\gamma>\gamma_0$ the following estimate holds: 
\begin{align}
\rho &< \bar u_\gamma(x) \leq  \lim_{r\arrow 0}\frac{1}{|B(x,r)|} \l|\int_{B(x,0)} \bar u_\gamma (y) - u^*(y)\, dy \r|+  \lim_{r\arrow 0}\frac{1}{|B(x,r)|} \int_{B(x,r)} u^*(y)dy \nonumber \\
& <   \varepsilon +  \lim_{r\arrow 0}\frac{1}{|B(x,r)|} \int_{B(x,0)} u^*(x)dy = \varepsilon + u^*(x).
 \end{align}
 %
%
%
Finally, taking $\varepsilon\arrow 0$ then $\rho \leq  \bar u^*(x)$. Analogously, we conclude that if $x\in\om$ is such that $\bar u_\gamma(x) <-\rho$ then $u^*(x) \leq -\rho$. Hence, property \eqref{eq:supp_u*} follows.
\end{proof}
}

{
\begin{theorem}\label{t:consistency}
 Let $(\bar u_\gamma )_{\gamma>0}$  a sequence of solutions of problem \eqref{e:OCP2}. Suppose that assumptions of Theorem \ref{t:exist} hold. There exists a subsequence $(\gamma_n)_{n\in \mathbb{N}}$ with $\gamma_n \arrow \infty$ as $n \arrow \infty$ and a subsequence $(\bar u_{n})_{n\in \mathbb{N}} \subset (\bar u_\gamma )_{\gamma>0}$ converging strongly in $L^2(\om)$ with $\bar u_n=\bar u_{\gamma_n}$  with limit $u^*$ in $U_{ad}$. Then, $u^*$ is a solution for problem \eqref{e:OCP} and the following convergence property is fulfilled:
  \begin{equation}
  \lim_{n \arrow \infty} J_{\gamma_n}(\bar u_n) = J(u^*) = \inf (P). 
  \nonumber
 \end{equation}   	
 \end{theorem}

\begin{proof}
 By Theorem \ref{t:opsys_u*} we consider $u^* \in L^2(\om)$, the weak limit of $(\bar u_n)_{n\in \mathbb{N}}$ (after extraction of a subsequence) which satisfies \eqref{eq:OPT*} and \eqref{eq:supp_u*}. Arguing as in Theorem \ref{t:exist}, we have that $u_n$ converges strongly in $L^2(\om)$. Thus, the optimality of   $\bar u_n$ implies that $J_\gamma(\bar u_n)  \leq J_\gamma (u)$ for any $u\in L^2(\om)$ and taking into account \eqref{eq:hub1}, it follows that 
\begin{align}
	J_{\gamma_n}(\bar u_n)  &\leq J_{\gamma_n} (u) \nonumber\\
							 & = ~\frac{1}{2}\| y-y_d \|^2_{L^2(\om)}+\frac{\alpha}{2}\|u\|^2_{L^2(\om)}+\beta \fsparse_{p,{\gamma}_n}(u) \nonumber \\
							 & \leq ~\frac{1}{2}\| y-y_d \|^2_{L^2(\om)}+\frac{\alpha}{2}\|u\|^2_{L^2(\om)}+\beta \fsparse(u) = J(u).  \label{eq:optim_0} 
\end{align}
  In particular, Lemma \ref{l:conv1} this implies that  
\begin{equation}
\liminf_{n\arrow 0} J_{\gamma_n}(\bar u_n)	= \lim_{n\arrow 0} J_{\gamma_n}(\bar u_n) \leq J(u^*).\label{eq:optim_1}
\end{equation}
Now, we argue the reverse inequality. We know that the sequence $(v_k)_{k\in \mathbb{N}}$ strongly converges to $u^*$ in $L^2(\om)$. Let us denote by $\bar y_n$ and $\bar\phi_n$ the state and adjoint state associated to $\bar u_n$, respectively, and by $y_{u_k}$ the state associated to $u_k$. By continuity of the quadratic terms and Lemma \ref{l:conv1}, we have that
\begin{align}
J_{\gamma_n}(u^*) & = ~\frac{1}{2}\| y^*-y_d \|^2_{L^2(\om)} + \frac{\alpha}{2} \norm{u^*}_{L^2(\om)}^2 +\fsparse_{p,{\gamma}_n} (u^*) \nonumber\\
	& =\lim_{k \arrow \infty} \sum_{n=k}^{N(n)} \frac{1}{2}\|  y_{n}-y_d \|^2_{L^2(\om)} + \frac{\alpha}{2} \norm{\bar u_n}_{L^2(\om)}^2 +\fsparse_{p,{\gamma_n}} (u_n) \nonumber\\
	&= \liminf_{n \arrow \infty} J_\gamma(\bar u_n). \label{eq:cost.1}
	\end{align}
Then, by Lemma \ref{l:uniconv_nonconvex}
 and \eqref{eq:cost.1} we see that 
$$J(u^*)\leq J_\gamma (u^*) \leq \liminf_{\gamma \arrow \infty}J_{\gamma_n} (\bar u_n). $$ This, together with \eqref{eq:optim_1} proves that $J(u^*) = \lim_{n \arrow 0} J_{\gamma_n}(\bar u_n)$, and from \eqref{eq:optim_0}
 we conclude that $u^*$ is  a minimum for problem \eqref{e:OCP}.
\end{proof}
}
 

\subsection{First--order necessary conditions with box--constraints} Since box--constraints are  important in applications, we give a further discussion when they are included in the optimal control problem \eqref{e:OCP}. Let us consider the set of feasible controls given by:
\begin{equation}\label{eq:U_ad}
	U_{ad}=\{u\in L^2(\om): u_a(x) \leq u(x) \leq u_b(x), \, \text{a.a. } x\in \om \},
\end{equation}  
where $u_a$ and $u_b$ are given functions in $L^\infty(\om)$ satisfying $u_a(x) < 0 < u_b(x)$ a.a. $x\in \om$. A similar analysis of existence of solutions and approximation of the regularized problems can be done with a few modifications of the associated results for the unconstrained case. 

The control constrained optimal control problem reads:
\begin{equation}
\tag{$P_C$} \label{e:OCP_c}
\begin{cases}
\displaystyle\min_{(y,u)} ~\frac{1}{2}\| y-y_d \|^2_{L^2(\om)}+\frac{\alpha}{2}\|u\|^2_{L^2(\om)}+\beta \fsparse_p(u)\\
\hbox{ subject to: }\\
u \in U_{ad} \qquad\text{  and}
\hspace{20pt}\begin{array}{rll}
A y=&u + f, &\hbox{in  } \om, \\
y=&0, &\hbox{on  }  \Gamma.
\end{array}
\end{cases}
\end{equation}

\begin{remark}
It follows by  definition \eqref{eq:U_ad} that $U_{ad}\subset B_\infty(0,M)$ with $M=\max\{\norm{u_a}_{L^\infty(\om)},\allowbreak\norm{u_b}_{L^\infty(\om)}\}.$ Therefore, according to Lemma \ref{l:null_sol} if $\beta>\beta_0= M^{\frac{p-1}{p}}\, \norm{S^{*}(Sf - y_d)}_{L^\infty(\om)}$ then $\bar u =0$ is solution of \eqref{e:OCP_c}.
\end{remark}

Analogous to the unconstrained optimal control problem \eqref{e:OCP'}, after introducing the control--to--state operator $S$ and replacing $\fsparse_p$ by $\fsparse_{p,\gamma}$, we introduce the regularized control constrained problem:

\begin{equation}
\tag{${P_{C}}_{\gamma}$} \label{e:OCP'_c}
\begin{cases}
\displaystyle\min_{(y,u)\in H_0^1(\om) \times L^2(\om)} ~\frac{1}{2}\| y-y_d \|^2_{L^2(\om)}+\frac{\alpha}{2}\|u\|^2_{L^2(\om)}+\beta \fsparse_{p,\gamma}(u)\\
\hbox{ subject to: }\\
u \in U_{ad} \qquad\text{  and}
\hspace{20pt}\begin{array}{rll}
A y=&u + f, &\hbox{in  } \om, \\
y=&0, &\hbox{on  }  \Gamma.
\end{array}
\end{cases}
\end{equation}

In the same fashion as the unconstrained problem, we define a DC representation of the cost functional for the constrained problem \eqref{e:OCP_c} by including the indicator function ${I}_{U_{ad}}$ for the admissible control set: 

\begin{align}\label{eq:GH_c}
\begin{array}{ll}
&\begin{array}{lrll}
	G: & L^2(\om) & \arrow &\reals \\
	   & u		& \mapsto &  G(u) : = \frac12 \norm{Su+Sf-y_d}^2_{L^2(\om)}	 +\alpha \norm{u}^2_{L^2(\om)} + \beta \delta_\gamma \norm{u}_{L^1(\om)} + I_{U_{ad}}, 
\end{array} \\
&\begin{array}{lrll}
	H: & L^2(\om) & \arrow &\reals \\
	   & u		& \mapsto &  H(u) : =  \beta \l( \delta_\gamma \norm{u}_{L^1(\om)} - \fsparse_{p,\gamma}(u) \r).
\end{array}
\end{array}
\end{align}
Thus, by similar arguments as in the unconstrained case and taking into account that $\partial I_{U_{ad}}(u)$ corresponds to the normal cone of $U_{ad}$ at $\bar u$, we can derive an analogous optimality system.

\begin{theorem}\label{t:fonc_c} Let $\bar u$ be a solution of \eqref{e:OCP'_c}. Then there exist $\bar y = S \bar u$ in $H_0^1(\om)$, 
an adjoint state $\bar\phi \in H_0^1(\om)$ and a multiplier $\zeta \in L^2(\om)$ and $\bar w$ given by \eqref{eq:dH} such that the following optimality system is satisfied :
\begin{subequations}
\begin{align}
& \begin{array}{rll}
A \bar y &= \bar u +f  & \text{in } \om, \\
	   \bar y & = 0 & \text{on } \ga, 	
\end{array}\label{eq:state1c}\\
& \begin{array}{rll}
A^* \bar \phi &= \bar y - y_d  & \text{in } \om, \\
	   \bar \phi & = 0 & \text{on } \ga, 	
\end{array} \label{eq:adj1c}\\
%
&\langle \bar \phi +\alpha \bar u + \beta \, (\delta_\gamma \, \zeta -\bar w ) , u - \bar u \rangle \geq  0, \quad \forall u \in U_{ad}	\label{eq:vic}\\
& \begin{array}{lll}
\zeta(x) &=1, & \text{ si } \bar u (x) >0, \\
\zeta(x) &=-1, & \text{ si } \bar u (x) <0, \\
|\zeta(x)| &\leq 1,  & \text{ si } \bar u (x) =0,
\end{array} \label{eq:zeta_mult}\\
& \;\text{ for almost all } x \in \om. \nonumber
\end{align}	
\end{subequations}
Moreover, there exist $\lambda_a$ and $\lambda_b$ in $L^2(\om)$ such that the last optimality system can be written as a KKT optimality system:
\begin{subequations}
\label{eq:OTP_P'c}
\begin{align}
& \begin{array}{rll}
A \bar y &= \bar u +f  & \text{in } \om, \\
	   \bar y & = 0 & \text{on } \ga, 	
\end{array}\label{eq:state1kkt}\\
& \begin{array}{rll}
A^* \bar \phi &= \bar y - y_d  & \text{in } \om, \\
	   \bar \phi & = 0 & \text{on } \ga, 	
\end{array} \label{eq:adj1kkt}\\
%
& \bar \phi +\alpha \bar u + \beta \, (\delta_\gamma \, \zeta -\bar w ) + \lambda_b - \lambda_a =0	\label{eq:gradientkkt}\\
& \begin{array}{ll}
	\lambda_a \geq 0, & \lambda_b \geq 0, \\
	\lambda_a (\bar u -u_a) =0, & \lambda_b (u_b-\bar u) =0,
\end{array}\\
& \begin{array}{lll}
\zeta(x) &=1 & \text{ si } \bar u (x) >0, \\
\zeta(x) &=-1 & \text{ si } \bar u (x) <0, \\
|\zeta(x)| &\leq 1  & \text{ si } \bar u (x) =0,
\end{array}
\end{align}
\end{subequations}
\end{theorem}
\begin{proof}
This theorem is proved by following the arguments of the proof of Theorem \ref{t:fonc}, where variational inequality \eqref{eq:vic} follows by taking into consideration classical results on convex analysis and the fact that $\bar w \in \nabla f(\bar u) + \beta\delta_\gamma\, \partial \norm{\cdot}_{L^1(\om)}(\bar u) + \partial I_{U_{ad}} (\bar u) $. 
\end{proof}

\vspace{4mm}
In addition, by the usual projection operator $\mathcal P_{U_{ad}}$ (see \cite[Lemma 1.11]{hinze2008}) on the admissible control set, the variational inequality \eqref{eq:vic} can be equivalently rewritten in equation form:
\begin{equation} \label{eq:PUad}
\bar u = \mathcal{P}_{U_{ad}}\l[ -\frac{1}{\alpha} \l( \bar \phi + \beta (\delta_\gamma \zeta-\bar w)\r)\r].
\end{equation}

\section{Numerical solution via the DC Algorithm (DCA)}\label{s:DCA}

In the former section we have derived necessary optimality conditions for problem \eqref{e:OCP2} and problem \eqref{e:OCP'_c}, essential to investigate the behavior of an optimal control. Moreover, these conditions are suitable for deriving numerical methods such as Semi-Smooth Newton method (SSN). 

By the nature of our problem we turn our attention to its numerical solution by adapting the DC algorithm. The application of the DC algorithm to our problem leads to a numerical scheme which relies on numerical methods for solving sparse $L^1$ optimal control problems, including SSN methods. Our method is completely determined by the formulation \eqref{e:OPT-DC} which is a suitable difference--of--convex functions representation of the original optimal control problem. We present the algorithm in a function space setting in the spirit of \cite{auchmuty89}.

The DC--Algorithm is based on the fact that: if $\bar u$ is the solution of the primal problem \eqref{e:OCP'} then  $\partial H (\bar u) \subset \partial G (\bar u)$ and conversely, if $u^*$ is the solution of the dual problem denoted by $({P'}^*)$ we have the inclusion $\partial G^* (u^*)\subset \partial H^* (u^*)$, where $H^*$ and $G^*$ correspond to the dual functions of $H$ and $G$ respectively. In \cite{auchmuty89}, an abstract framework for the DC algorithm in Banach spaces is presented. Although the functions $G$ and $H$ do not satisfy all assumptions in \cite{auchmuty89}, some of the results in  \cite{auchmuty89}  can be extended to our case with slight modifications. In particular,  if we define the function 
\begin{equation}
	L(u,w) = ( w, u )_{L^2(\om)} -G^*(w) - H(u),
\end{equation}
then, according to \cite{auchmuty83}, this is a \emph{Lagrangean of type I},  which we use to interpret optimality conditions for \eqref{e:OCP'} in terms of $L$. Indeed, if $\bar w \in \partial H(\bar u) \subset \partial G(\bar u)$, then we have that  $\bar u \in \partial G^*(\bar w)$. This is equivalent to the following condition:
\begin{subequations}\label{eq:saddle.1}
\begin{align}
	L(\bar u, w ) \geq L(\bar u, \bar w ), \\
	L( u, \bar w ) \geq L(\bar u, \bar w ),
\end{align}
\end{subequations}
for all $u$ and all $w$ in $L^2(\om)$. The pair $(\bar u, \bar w)$ is referred as $\partial$--critical point of $L$, see \cite{auchmuty89}.

This symmetry means that DC--Algorithm alternates in computing approximations of the solutions for the primal and the dual problems as follows:
\begin{subequations}\label{eq:inclusions}
\begin{align}
	\text{First chose: }& w_k \in \partial H(u_k), \\
	\text{then chose: }& u_k \in \partial G^*(w_k). 
\end{align}	
\end{subequations}
A more detailed discussion on the DC method can be found in \cite{dinh2014} and \cite{auchmuty89}. In particular, in \cite{auchmuty89} the authors study the convergence properties for the DC algorithm in abstract spaces that covers our case with small changes.

Let us give a precise meaning to the numerical problems generated by \eqref{eq:inclusions}. In view of the identity $\partial H ( u_k) = \{ w_k \}$, formula \eqref{eq:dH} implies that  $w_k$ is given by 
\begin{equation}\label{eq:w_k}
w_k = \begin{cases}
0, \hspace{7cm}\text{if }\quad  |u_k(x)| \leq \frac1\gamma, \\
\left[ \delta_\gamma - \frac{1}{p} \left(|u_k(x)| +\frac1\gamma\frac{1-p}{p} \right)^{\frac{1-p}{p}}\right] \sign ( u_k(x)), \quad\text{ otherwise.}
\end{cases}
\end{equation}

On the other hand, for a convex and lower semi--continuous function $g$, it follows that
$$ g(x) = \sup \{ \langle x,  y\rangle - g^*(y) \}. $$
Moreover,  according to Rockafellar \cite{rockafellar2015} the subgradients can be computed as: 

\begin{align}
\ds
	\partial G (y) =\text{argmax}_{w} \{ \langle y , w \rangle  - G^{*}(w) \}, \label{eq:subG} \\
	\partial G^* (w) =\text{argmax}_{z} \{ \langle w , z \rangle - G(z) \},     \label{eq:subG*}
\end{align}
therefore, $u_k$ can be obtained by solving the following optimal control problem
\begin{equation}\label{eq:L1-subproblem}
\min_{u_{k+1}} \frac12 \norm{S u_{k+1} + Sf - y_d}^2_{L^2(\om)} + \frac{\alpha}{2}\norm{u_{k+1}}^2_{L^2(\om)} + \delta_\gamma \beta \norm{u_{k+1}}_{L^1(\om)}	- \int_{\om} w_k u_{k+1} \,dx.
\end{equation}

In case of the presence of box--constraints on the control, our formulation yields a box--constrained $L^1$ optimal control subproblem

\begin{align}\label{eq:L1-subproblem-c}
\min_{u_{k+1} \in U_{ad}} &\frac12 \norm{S u_{k+1} - y_d}^2_{L^2(\om)} + \frac{\alpha}{2}\norm{u_{k+1}}^2_{L^2(\om)} + \delta_\gamma \beta \norm{u_{k+1}}_{L^1(\om)}	- \int_{\om} w_k u_{k+1} \,dx. 
\end{align}

\begin{remark}
	By the form of the DC splitting we replace problem \eqref{eq:subG} by the direct computation of $w_k$ from formula \eqref{eq:w_k}. In addition, observe that problem \eqref{eq:L1-subproblem} is a convex $L^1$--sparse optimal control problem with penalization parameter $\delta_\gamma\beta$, for which it is known to have a unique solution for $\alpha>0$ c.f. \cite{stadler09}. The case of $\alpha=0$ with box--constraints is also possible. Moreover, this problem can be solved numerically in an efficient way. For example, it can be solved by semi--smooth Newton methods proposed in \cite{stadler09} or, it can be solved in the framework of sparse programming problems in finite dimensions after its discretization.  
\end{remark}

{
In order to complete the presentation of our algorithm, we now turn our attention to the following as stopping criterion. Looking at the gradient equation \eqref{eq:gradeq} we could consider checking approximately that 
\begin{equation}\label{eq:stoppping}
	\zeta_k=\frac{1}{\beta\delta_\gamma}\l( w_k- \phi_k - \alpha  u_k  \r)    \in \partial\, \norm{\cdot}_{L^1(\om)}(u_k) ,
\end{equation}  
where $u_k$, $\phi_k$, $w_k$ represent the corresponding approximations of the optimal control, the adjoint state and the multipliers in the $k$--th iteration. This guarantees that the associated quantities satisfy the optimality system.  Other stopping criteria can be also used. For example,  condition \eqref{eq:saddle.1} can also be checked for stopping the algorithm.
}

\begin{algorithm}[H]
\caption{DCA for problem \eqref{e:OCP2}}
\begin{algorithmic}[1]
\STATE Initialize  $u^0$.
\WHILE{\hbox{stoping criteria is false}}
\STATE Compute $w_k$ given by \eqref{eq:w_k}
\STATE Compute $u_{k+1}$  by solving problem \eqref{eq:L1-subproblem} or \eqref{eq:L1-subproblem-c} in case of control constraints.
\STATE $k \gets k+1$.
\ENDWHILE
\end{algorithmic}\label{alg:DCA1}
\end{algorithm}



\subsection{Advantages and disadvantages of DC--Algorithm}

{
Algorithm \ref{alg:DCA1} is a first--order method, which provides a primal--dual updating procedure without a line--search step. 
Although in the proposed DC method there is no need of a line--search procedure, the methods to solve the inner subproblem related with the computation of $\partial G^*$ are not PDE--free. In fact, the computational cost is concentrated in solving the subproblem, and the algorithm used to solve it might still require a line--search procedure. However, this is not an explicit feature of DCA.

 {
 In order to solve the $L^1$--norm subproblem \eqref{eq:L1-subproblem} one may apply different methods available in literature. See for example \cite{wright12} for a survey of methods for these type of problems. In particular, we might apply the numerical scheme developed in \cite{stadler09}.  

As an alternative to the DC--algorithms discussed here, the optimality system \eqref{eq:optim_1} obtained from the DC representation of the problem can be used as a basis for the derivation of semi--smooth Newton schemes. Thus, we can take advantage of this optimality system in order to obtain superlinear methods based on the DC--approach, see \cite{itoku2008} for a complete review of such methods.  It is also worth taking into account that SSN schemes will require solving a coupled system involving the state $y$, the adjoint state $\phi$, and the multipliers $\zeta$, $\lambda_a$ and $\lambda_b$, resulting in a large system of equations which is usual in Newton methods. Solving this system might be computationally demanding because of its size, which depends on the discretization and the dimension of the domain for all the coupled variables. In the DC method we still have to solve PDEs but, in contrast to SSN, the systems involved are not coupled. Despite of this, the first order nature of DCA will demand more iterations for converging to an approximated solution.} 
We summarize the numerical properties of the algorithm in the following table.

\begin{table}[hbt]
\begin{center}
  \begin{tabular}{lll}

 Issues  & PDA &  DCA\\
    \hline   \hline
    Iterative subproblem    & Linear system & $\norm{\cdot}_{L^1}$ optimal control problem \\
 Linear systems   & Large, sparse &  Dense, small (using OESOM solver)\\
 					& $(y,u,\phi)$ & ($u$ only) \\
 Sparsity & $\approx 0$ &  $=0$\\
 Tuning parameters & 1			&  4   \\
 \hline
  \end{tabular}
  \caption{Numerical properties of PDA (proposed in \cite{itoku2014}) and DCA}
  \end{center}
\end{table}

Whereas  in the primal dual algorithm proposed in \cite{itoku2014} a large sparse linear system for $(y,u,\phi)$ needs to be solved in each iteration, the DC--Algorithm requires the solution of a sparse optimal control problem iteratively. In our setting we chose to use descent methods, intended specifically for $L^1$--sparse problems (\cite{dlrlm2017}). Note also that, in each iteration, OESOM solves a dense linear system depending only on the inactive components of the control variable, which can be decoupled from the active ones. By construction, the approximated Hessian is dense and its construction is based on the BFGS matrix.

 Using tailored methods for solving $L^1$--sparse problems has the advantage of recovering the sparse components of the solution, as discretization points with vanishing control. In contrast, PDA computes sparsity only approximately close to 0. See Figure \ref{fig:ex2_sparsity}.

We also observe that the DC--algorithm requires the tuning of more parameters, also depending on the method used to solve the inner subproblem. This is a drawback when compared with PDA, which requires choosing only one regularization parameter $\varepsilon$.  

Finally, we mention that both algorithms can be combined. For example, after obtaining a solution using PDA we can recover the sparse components by refining the solution as the input for DCA.
}

\section{Implementation aspects}

\subsection{Approximation} For simplicity, the approximation of problems \eqref{e:OCP} and \eqref{e:OCP_c} is done by the finite--difference scheme,  although other discretization methods might be applied as well. Uniform meshes are considered in the domain $\om$ with $N$ internal nodes. The associated mesh parameter is given by $h=\frac{1}{N+1}$. Then, the state equation \eqref{eq:elliptic} is solved numerically with the finite difference method while the approximation of the integrals is computed using the following mid--point rule:

\begin{align} \label{eq:quadrature}
\int_a^b\int_c^d u(x,y) dydx \approx  \frac{1}{4}h^2&{}\big\{u(a,c)+u(b,c)+u(a,d)+u(b,d)\\ &{}+2\sum_{i=1}^{n-2}u(x_i,c)+ 2\sum_{i=1}^{n-2}u(x_i,d) + 2\sum_{i=1}^{n-2}u(a,y_i) \nonumber \\
&{}+2\sum_{i=1}^{n-2}u(b,y_i)+4\sum_{i=1}^{n-2} \sum_{j=1}^{n-2}u(x_i,y_i)\big\}. \nonumber	
\end{align}
Using this approximation, and reshaping the matrix $(u(x_i,y_j))_{i,j=1,\ldots,N}$ as a vector $\mathbf u \in \reals^{N^2}$ the $L^1$--norm is approximated by 
\begin{equation}
	\norm{u}_1 \approx \sum_{i=1}^{N^2} c_i |\mathbf{u}_i|,
\end{equation} 
where the $c_i$'s are the corresponding coefficients given by \eqref{eq:quadrature}.
\subsection{Auxiliar $L^1$--sparse optimal control problems} DC--algorithm \ref{alg:DCA1} has a simple structure. However, the method requires to solve auxiliar $L^1$--norm optimal control subproblems \eqref{eq:L1-subproblem} (or \eqref{eq:L1-subproblem-c} in the constrained case). Clearly, the efficiency of the proposed algorithms strongly depends on the numerical methods applied for solving \eqref{eq:L1-subproblem} and \eqref{eq:L1-subproblem-c}. As mentioned earlier, the numerical solution of the $L^1$--norm optimal control problems can be done by semi--smooth Newton methods as in \cite{stadler09}. However, semi--smooth Newton methods do not guarantee a reduction in the cost function in each iteration. 
%

In the current numerical scheme,  the application of numerical methods for solving $L^1$--sparse problems is straightforward. Indeed, we only need to provide the cost function and the corresponding gradient which involves the computing of the adjoint state. The last one can be evaluated by means of the adjoint state \eqref{eq:adj1}. Several methods involve second order information of the smooth part of the cost, which can be considered using Hessians or its approximation by means of BFGS or LBFGS methods. In addition, approximated second order information of non differentiable term is calculated by the built--in enriched second order information constructed by the OESOM algorithm using weak derivatives of the $L^1$--norm, see \cite{dlrlm2017} for details.

\section{Numerical Experiments}

In order to investigate the numerical performance of the proposed DC--algorithm in Section \ref{s:DCA} we have implemented Algorithm \ref{alg:DCA1} using MATLAB. The associated sparse $L^1$ subproblem was solved using the OESOM algorithm \cite{dlrlm2017} by extending it to the box--constrained case with an additional projection step on the admissible control set. 

As illustrative examples, we consider the following tests defined on the unit square domain $\Omega = (0,1) \times (0,1)$. 

\begin{example}\label{ex:1} We consider problem \eqref{e:OCP} for $A=-\Delta$ and  $y_d = e^{-\cos (2 \pi x y)^2/0.1}$.

\subsection*{Performance of a single run} We first solve this example fixing the values of $\alpha=1/4$ and $\beta = 7/10$. Algorithm \ref{alg:DCA1} gives an approximated solution after 18 iterations stopping when the quantity: $-\frac{1}{\beta \delta_\gamma}(\alpha\bar u_k + \bar \phi_k - \beta \bar w_k)$ belongs to $\partial \norm{\cdot}_{L^1(\om)}(u_k)$. The table and graphics below show the performance and behavior of DCA. We observe in Figure \ref{fig:ex1_convergence1}, with logarithmic scale in the $x$ axis, the decreasing behavior of the objective function is more intensive in the first iterations. We also show the decreasing of the distance of consecutive approximated multipliers in the logarithmic scale in the $y$ axis. 
\begin{figure}[ht!]
\centering
\begin{subfigure}{0.45\textwidth}
\includegraphics{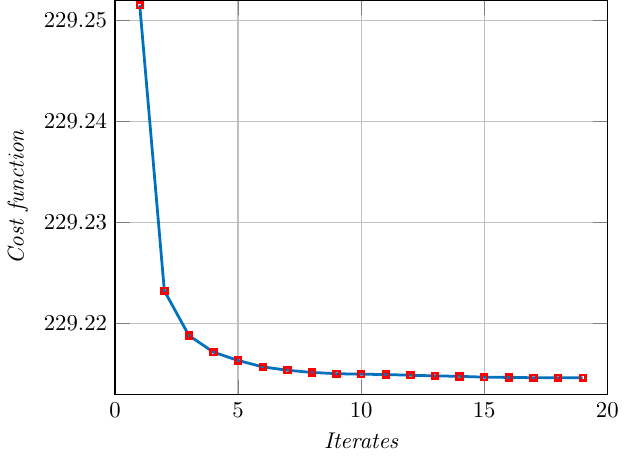}
\end{subfigure}
\begin{subfigure}{0.45\textwidth}
\includegraphics{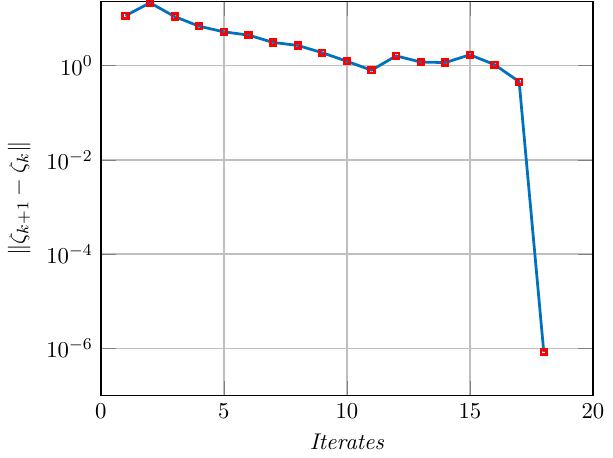}
\end{subfigure}
\caption{Cost function and size of consecutive values of $\zeta$ at $\beta=0.004$}
 \label{fig:ex1_convergence1}
\end{figure}
 
Figure \ref{fig:ex1_convergence1} (right) depicts the evolution of stopping criteria, which is more erratic with a decreasing tendency. In each iteration new sparse components appear then, when comparing consecutive multipliers, they may differ from 0 to 1 in those components, causing oscillations on their difference. We also realize in Table \ref{tab:performance1} that the number of sparse components
 of the approximated solution is increasing at every iterate.

\begin{table}[hbt]
\begin{center}
{ \normalfont
  \footnotesize
\begin{tabular}{lllllll}
$k$ & Cost & Residual & $\|\zeta_{k+1} - \zeta_k \|$ & Null &  OESOM & Execution \\
& & & &entries & iterations &time (s)  \\
\hline
\hline
1&229.2515&63.2346&0.044062&42&4&1.4507\\\hline
2&229.2233&0.28763&11.5196&893&11&6.1601\\\hline
3&229.2188&0.074112&21.7329&1170&11&6.303\\\hline
4&229.2172&0.037194&11.0554&1303&8&4.7229\\\hline
5&229.2164&0.026078&6.8869&1365&11&5.7958\\\hline
6&229.2157&0.019197&5.2423&1423&23&9.8522\\\hline
7&229.2154&0.014521&4.4751&1455&6&3.697\\\hline
8&229.2152&0.010528&3.1231&1471&6&3.3486\\\hline
9&229.2151&0.0074617&2.6985&1487&5&3.0378\\\hline
10&229.215&0.0046823&1.8893&1493&7&3.5644\\\hline
11&229.215&0.0057956&1.2435&1495&7&3.1689\\\hline
12&229.2149&0.0071222&0.80207&1501&6&3.573\\\hline
13&229.2148&0.0056617&1.6193&1507&6&2.9722\\\hline
14&229.2148&0.0054029&1.1993&1511&7&3.7972\\\hline
15&229.2147&0.0060645&1.1654&1519&6&3.1879\\\hline
16&229.2147&0.0041252&1.7088&1521&6&2.7979\\\hline
17&229.2147&0.0011533&1.0534&1525&6&2.7833\\\hline
18&229.2147&0.000409&0.45796&1525&5&2.0453\\\hline
\end{tabular}
\caption{Performance data for DCA for Example 1}
\label{tab:performance1}
}
\end{center}
\end{table}
\subsection*{Varying the regularization parameter $\gamma$} According to our theory, it is expected that if $\gamma \arrow \infty$  the solution $\bar u_\gamma \arrow \bar u$. Here, we solve Example \ref{ex:1} for increasing values of $\gamma$. The numerical evidence of this convergence behavior is reflected in Table \ref{tab:vgamma} where we observe optimal cost converges to a fixed value, whereas sparsity also stabilizes at 1525 null components of the solution. 
\begin{table}[hbt]
\begin{center}
{ \normalfont
  \footnotesize
\begin{tabular}{cccc}
$\gamma$ & Optimal & Sparse & DCA \\
         & Cost    & components &  Iterations\\
\hline
\hline
100&229.219724&1080&17\\\hline
200&229.214857&1499&24\\\hline
500&229.214082&1582&18\\\hline
1000&229.214356&1553&22\\\hline
1500&229.214650&1525&15\\\hline
2000&229.214651&1525&17\\\hline
2500&229.214650&1525&20\\\hline
3000&229.214650&1525&20\\\hline
4000&229.214650&1525&22\\\hline
5000&229.214650&1525&24\\\hline
\end{tabular}
\caption{Numerical convergence for increasing values of $\gamma$.}
\label{tab:vgamma}
}
\end{center}
\end{table}

\subsection*{Varying the regularization parameter $\beta$} Now we experiment with different values of $\beta$, which determines the sparsity--inducting term $\fsparse$. Table \ref{tab:vbeta} shows that larger values of $\beta$ result in sparser solutions until the solution vanishes, which illustrates Lemma \ref{l:null_sol}. As expected, it can also be observed that the optimal cost increases according to the sparsity of the solution, reflected in smaller supports of the controls.
\begin{table}[htb]
\begin{center}
{ \normalfont
\footnotesize
\begin{tabular}{cccc}
$\beta$ & Optimal & Sparse & DCA \\
         & Cost    & components &  Iterations\\
\hline
\hline
0.0002&229.1145&1034&25\\\hline
0.0005&229.259&1729&30\\\hline
0.0010&229.4327&2528&37\\\hline
0.0015&229.5503&3004&30\\\hline
0.0020&229.6252&3359&31\\\hline
0.0025&229.6676&3631&40\\\hline
0.0030&229.6849&3843&37\\\hline
\end{tabular}
\caption{ Solutions become sparser as  $\beta$ increases.}
\label{tab:vbeta}
}
\end{center}
\end{table}
\begin{figure}[htb!]
                \centering
                \begin{tabular}{cc}
%
                \includegraphics{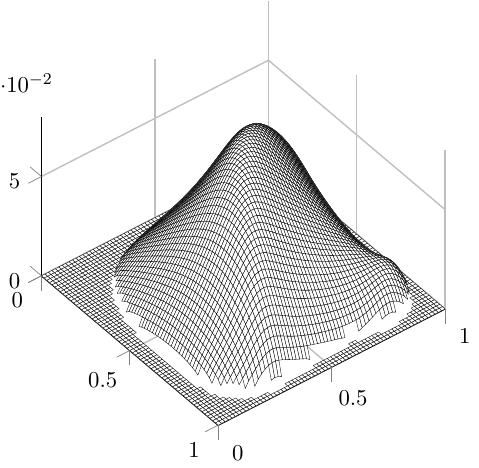} & 
                \includegraphics{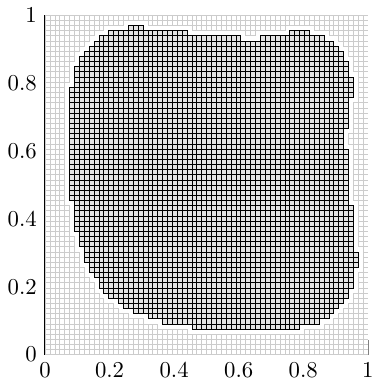}
                \end{tabular}
                \caption{Optimal control and its support for $\beta=0.0002$.}
                \label{fig:exp1_c01}
\end{figure}%
\begin{figure}[htb!]
                \centering
                \begin{tabular}{cc}
                 \includegraphics{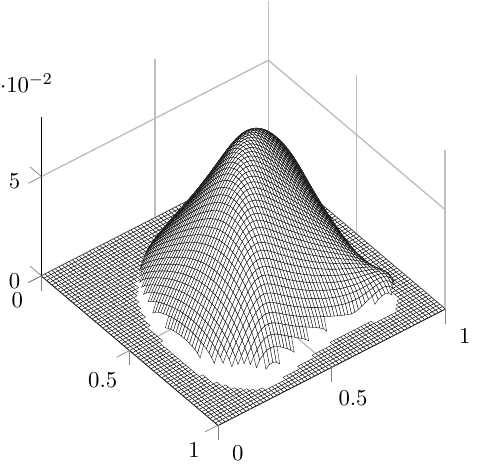} &
                 \includegraphics{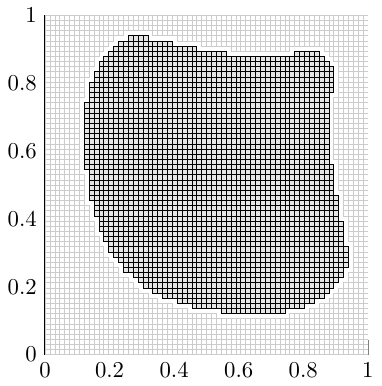}
                \end{tabular}
                \caption{Optimal control and its support for $\beta=0.001$.}
                \label{fig:exp1_c02}
        \end{figure}%
        
\begin{figure}[htb!]
                \centering
                \begin{tabular}{cc}
                \includegraphics{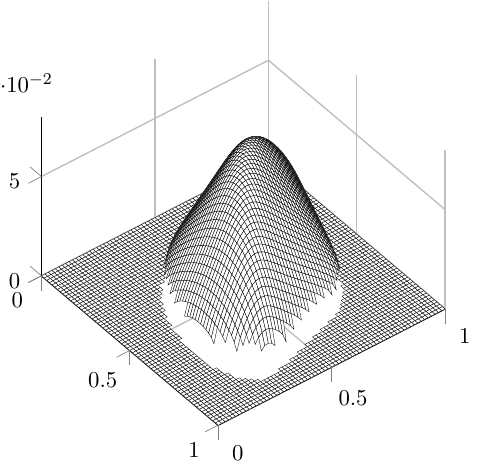} & 
                \includegraphics{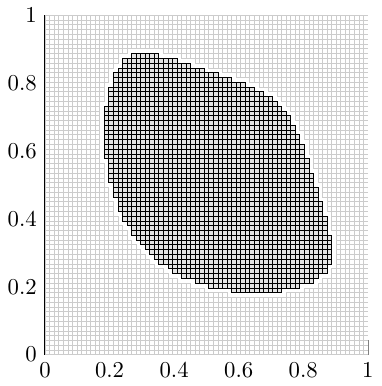}
                \end{tabular}
                \caption{Optimal control and its support for $\beta=0.002$.}
 				\label{fig:exp1_c03}
\end{figure}
\begin{figure}[htb!]	                
                \centering
                \begin{tabular}{cc}
                \includegraphics{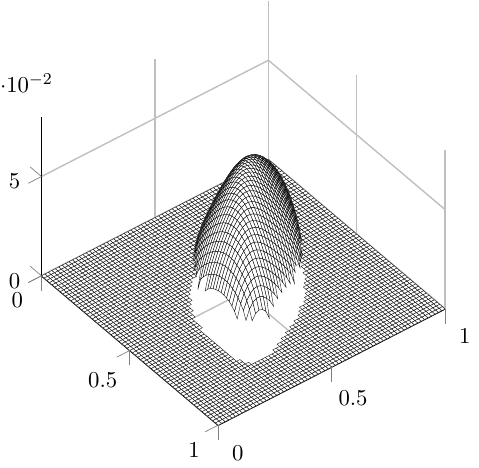} & 
                \includegraphics{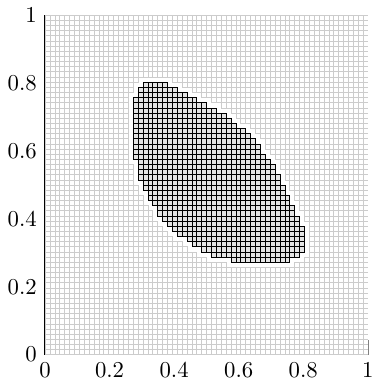}               
                \end{tabular}
               \caption{Optimal control and its support for $\beta=0.003$.}
                \label{fig:exp1_c04}      
\end{figure}
\subsection*{Varying the exponent $p$} We finish this example with the variation of the fractional exponent $1/p$ which also plays a role in the sparsity of the solution. In fact, $p$ determines how expensive is a sparse control. It is known that for larger values of $p$ the sparsity term tends to produce a volume constraint induced by the Donoho's counting norm cf.\cite{itoku2014}. However, the increment of $p$ does not necessarily increase sparsity in the solution as we can see in Table \ref{tab:vp}. 
\begin{table}[hbt]
\begin{center}
{ \normalfont
\footnotesize
\begin{tabular}{cccc}
$p$ & Optimal & Sparse & DCA \\
         & Cost    & components &  Iterations\\
\hline
\hline
1&229.2028&789&4\\\hline
1.2&229.3232&1860&18\\\hline
1.5&229.4736&2778&32\\\hline
2&229.6256&3355&27\\\hline
4&229.8814&3667&26\\\hline
8&230.3485&3441&23\\\hline
10&230.5699&3323&28\\\hline
20&231.4921&2846&23\\\hline
\end{tabular}
\caption{Influence of the power parameter $p$  in the sparsity  of the solution.}
\label{tab:vp}	
}
\end{center}
\end{table}
\clearpage
\end{example}
\begin{example}\label{e:comp1} In this example, we compare DC--Algorithm with the primal--dual method proposed in \cite{itoku2014}[See eq. (5.7), pg. 1273 for problem $(P_{s,\varepsilon})$] developed to solve optimal control problems involving $L^{q}$--penalizations with $q \in(0,1)$. Here, we consider an additional $L^2$ penalization on the gradient of the control. Therefore, the control space is restricted to a subset of $H_0^1(\om)$. Although, this penalization is beyond our theory, it can be considered with straightforward modifications. The problem reads
\begin{equation}
\tag{E2} \label{eq:ex1}
\begin{cases}
\displaystyle\min_{(y,u)} ~\frac{1}{2}\| y-y_d \|^2_{L^2(\om)}+\frac{1}{2}\|\nabla u\|^2_{L^2(\om)}+\beta \fsparse_2(u)\\
\hbox{ subject to }\\
\hspace{40pt}\begin{array}{rll}
- \Delta y=&u, &\hbox{in  } \om, \\
y=&0, &\hbox{on  }  \Gamma,
\end{array}
\end{cases}
\end{equation}	
In the framework of \cite{itoku2014}, we choose the quantities $B=I$, $E=-\Delta$, $K=E^{-1}$, $g=0$, $f=y_d$ and $Y=L^2(\om)$. Therefore, the numerical scheme  (5.7) in \cite{itoku2014} consist in the sequence of equations of the form:
    \begin{equation}\label{eq:itoku_scheme}
    	- \Delta u_{k+1} +  K^*K u_{k+1} + \frac{\beta / p}{\max{(\varepsilon^{2-\frac{1}{p}}}, |u_{k}|^{2-\frac{1}{p}})  } u_{k+1}= K^*y_d, \quad k=0,1,2,\ldots
    \end{equation}
The operator $K$  and $K^*$ involve the inverse of the differential operator (corresponding to the laplacian, in this example). However, it is an uncommon situation having an explicit representation of $K$ and $K^*$, rather we have to solve the associated PDE. Therefore, we introduce the state  $y_{k+1}$ and the adjoint state $\phi_{k+1}$. Then, equation \eqref{eq:itoku_scheme}  is reformulated as the following iterative system:
\begin{equation}\label{eq:itokunisch}
\l(
\begin{array}{ccc}
\alpha E + \frac{\beta / p}{\max{(\varepsilon^{2-{1}/{p}}}, |u_{k}|^{2-{1}/{p}})  } & I & 0 \\
0 & E & -I \\
-I & 0 & E	
\end{array}
\r) 
\l(
\begin{array}{c}
	u_{k+1} \\
	\phi_{k+1} \\
	y_{k+1}

\end{array}
\r) 
=
\l(
\begin{array}{c}
0 \\
y_d	\\
0
\end{array}
\r). 
\end{equation}
In order to compare DC-Algorithm (DCA) with the \emph{Primal Dual} based Algorithm \eqref{eq:itokunisch} (which we will refer as PD-Algorithm, PDA for short) we observe their performance at different values of the regularization parameters with both methods starting from the same initial point $u_0$. 
There is not a direct relation between the regularization parameters $\varepsilon$ of PD-Algorithm and $\gamma$ used in DC-Algorithm. Therefore,  we chose regularization parameters for each regularizer such that the function $|t|^{1/p}$ with approximately the same error, i.e.  the regularization error satisfies: $R_e=\norm{ |t|^{1/p} - t_r }_\infty \approx \text{\emph{tol}}$, where \emph{tol} is a tolerance and $t_r$ denotes the regularization. 

Moreover, since both algorithms have different stopping rules  we observe the cost value after 100 iterations, to guarantee that booth algorithms are close enough to the solution. The results are summarized in Table \ref{tab:comp1}. In our experiments we found a similar performance of both algorithms. After 100 iterations we observe that PDA or DCA can reach the minimum cost, depending on the regularization parameters.

\begin{table}[ht]
\centering
\scriptsize
\begin{tabular}{cllccc}

                         & Reg.              & Re             & Cost ($\beta=0.005$) & Cost($\beta=0.01$)   &   Cost ($\beta=0.2$)  
\\ \hline \hline
\multirow{2}{*}{PD-Algorithm}   & $\varepsilon = 0.0001$		& 0.00750		  & 6.10089 & 7.2873 & 8.30983\\
								&  $\varepsilon = 0.001$ 		& 0.01290 		  & \textbf{6.10048}& 7.2626 & 8.31011\\
\hline								
\multirow{2}{*}{DC-Algorithm}  &  $\gamma =500$ 				& 0.00746         & 6.10053 & \textbf{7.1810}&\textbf{8.19439}		\\
 				&$\gamma =300$ 				& 0.01298           & 6.10050         &  \textbf{7.1315} & 8.19786\\
   \hline
\end{tabular}
\caption{Comparison with primal--dual algorithm after 100 iterations}
\label{tab:comp1}
\end{table}

Figure \ref{fig:ex2_sparsity} shows sparse components of the solution computed by PDA and DCA methods respectively. It can be observed that PD--Algorithm computes sparse components approximately 0 ($\approx 10^{-4}$) while DC--Algorithm is able to recover zero sparse components as expected from the theory.

\begin{figure}[hbt!]	                
                \centering
                \begin{tabular}{cc}
                \includegraphics{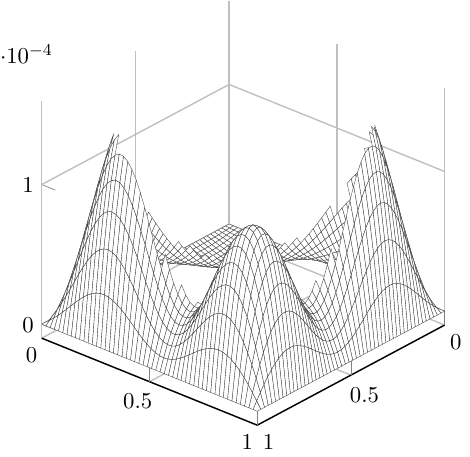} & 
                \includegraphics{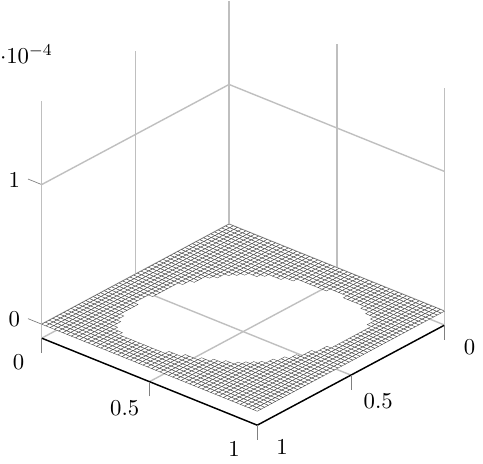}               
                \end{tabular}
                \caption{Sparse part of the control computed by PD--algorithm (left) and  DC--Algorithm (right)} 
                \label{fig:ex2_sparsity}
\end{figure}

\end{example}

\begin{example} This example consists in imposing box--constraints on Example \ref{ex:1}. We keep the same parameters as in Example \ref{ex:1}. Therefore, we require in addition that 
\begin{equation*}
	u \in U_{ad}=\{u \in L^2(\om) : 0 \leq u \leq 0.035 \}.
\end{equation*} 
Similar  results are observed in this case as depicted in Figure \ref{fig:ex2_convergence}. The structure of the sparsity and the support of the optimal control is similar but in this case the optimal control is also active on the prescribed bounds as  observed  in Figure \ref{fig:exp2_c01}.
 \begin{figure}[htb!]
\centering
\begin{subfigure}{0.45\textwidth}
\includegraphics{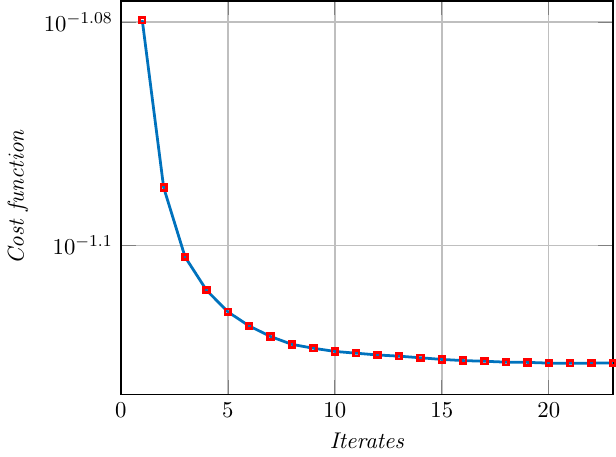}
\end{subfigure}
\begin{subfigure}{0.45\textwidth}
\includegraphics{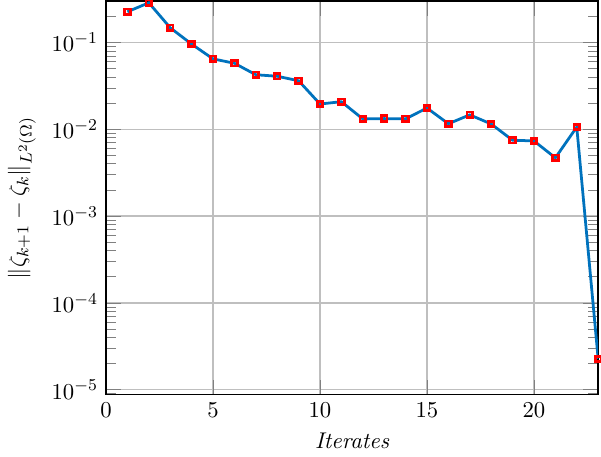}
\end{subfigure}
%
%
%
%
%
\caption{Cost function and distance of consecutive solutions for the multiplier $\zeta$.}
 \label{fig:ex2_convergence}
 \end{figure}
\begin{figure}[h!]	                
                \centering
                \begin{tabular}{cc}
                \includegraphics{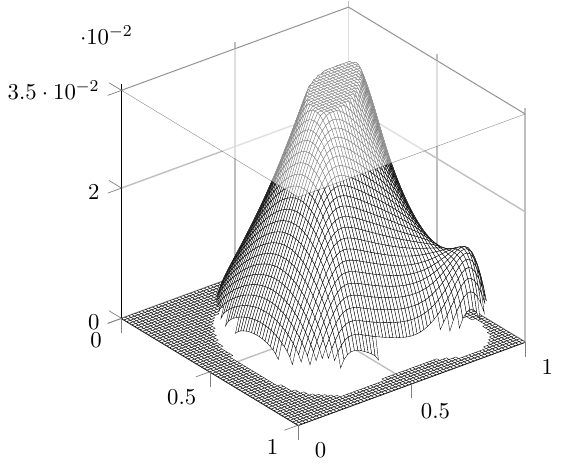} & 
                \includegraphics{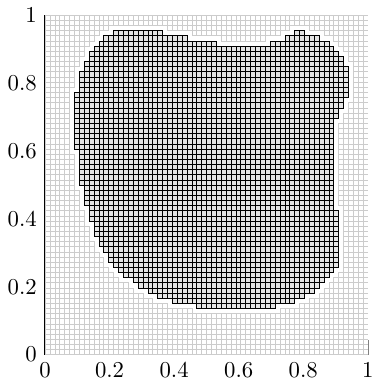}               
                \end{tabular}
                \caption{Box--constrained optimal control and its support.}
                \label{fig:exp2_c01}
\end{figure}

Our final experiment is out of scope of this paper since our theory does not consider the case $\alpha=0$. However, the method is still useful to this case and further analysis is required. Our problem consists in a box--constrained optimal control problem with $L^q$--term only ($\alpha=0$). Here the desired state is $y_d(x_1,x_2) = \sin(2\pi x_1)\, \sin(2 \pi x_2)$ and the set of admissible controls is given by 
\begin{equation*}
	u \in U_{ad}=\{u \in L^2(\om) : -0.035 \leq u \leq 0.035 \}.
\end{equation*}

  In this case we observe a typical shape of a bang--bang optimal control (see Figure \ref{fig:exp2_c02}). 

\begin{figure}[ht!]	                
                \centering
                \begin{tabular}{cc}
                \includegraphics{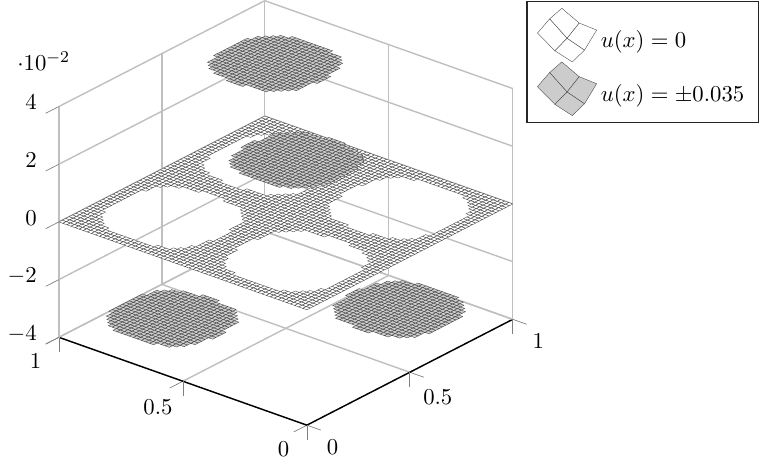} & 
                \includegraphics{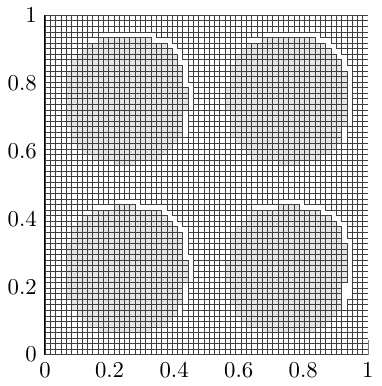}               
                \end{tabular}
                \caption{Box--constrained optimal control and its support for $\alpha=0$.}
                \label{fig:exp2_c02}
\end{figure}

\end{example}
\section{Conclusions}
We were able to apply the DC methodology to optimal control problems involving the \(L^q\) ($0<q<1$) nonconvex terms by introducing a Huber like smoothing for \(L^q\) quasinorms. The proposed smoothing captures the nonconvex nature and nondifferentiability of the \(L^q\) terms which are reflected in the computed approximated solutions.

Using the proposed Huber regularization we have identified a suitable representation of the cost as a difference--of--convex functions. This is crucial for an efficient application of the DC algorithm to our problem since one of the convex parts ($H$) is G\^ateaux differentiable and therefore its subgradient is computed directly.  For the other convex function ($G$), the computation of its subgradient needs to iteratively solve a sparse optimal control problem with $L^1$ penalization, for which there are efficient methods at hand.

Furthermore, the DC approach is helpful for deriving first--order necessary optimality conditions. The obtained optimality system is an important result since it provides deeper insight in understanding the nature of the solutions for this class of optimal control problems.

The proposed DC algorithm solves the nonconvex optimal control problem efficiently as shown in the numerical examples section. Although it is known that the DC algorithm is of first order, the question of the rate of convergence in this setting remains to be answered.

 Our algorithm can compute approximate solutions which reveal the sparse structure of the optimal controls. In addition, the optimality system  derived by the DC approach is suitable for semismooth Newton methods (SSN). However, the application of second order methods, such as SSN, requires further research.
\clearpage


\end{document}